\documentclass{amsart}

\usepackage{subfig}
\usepackage{url}
\usepackage[logical-quotes,nobysame]{amsrefs}
\usepackage{amssymb,amsmath,amstext,amsopn}
\usepackage{listings}
\usepackage{graphicx}
\long\def\beginpgfgraphicnamed#1#2\endpgfgraphicnamed{\includegraphics{#1.pdf}}
\newtheorem{theorem}{Theorem}[section]

\newtheorem{lemma}[theorem]{Lemma}

\newtheorem{conjecture}[theorem]{Conjecture}
\newtheorem{definition}[theorem]{Definition}
\newtheorem{remark}[theorem]{Remark}

\numberwithin{equation}{section}

\newcommand{\Real}{{\mathbb{R}}}
\newcommand{\Rd}{ {\Real^d} }

\title[Inequalities for eigenvalues of triangles]{Isoperimetric inequalities for eigenvalues of triangles}

\author{Bart{\l}omiej Siudeja}

\address{Department of Mathematics, Purdue University, West Lafayette, Indiana 47906}

\email{siudeja@math.purdue.edu}



\begin{document}

\begin{abstract}
Lower bounds estimates are proved for the first eigenvalue for the Dirichlet Laplacian on arbitrary triangles using various symmetrization techniques.
These results  can viewed as a generalization of P\'olya's isoperimetric bounds. It is also shown that amongst triangles, the equilateral triangle minimizes the spectral gap and (under additional assumption) the ratio of the first two eigenvalues. This last result resembles the Payne-P\'olya-Weinberger conjecture proved by Ashbaugh and Benguria.
\end{abstract}

\maketitle
\section{Introduction}
\noindent
The purpose of this paper is to prove new isoperimetric--type  bounds for the
eigenvalues of the Dirichlet Laplacian on arbitrary triangles. Given a domain $D$ we will use 
``eigenvalue of the domain $D$"  to refer to an eigenvalue of the Dirichlet Laplacian on the domain $D$.
The eigenvalues for a domain form a nondecreasing sequence $\{\lambda_i, j\geq 1\}$ with 
$\lambda_1<\lambda_2$. Throughout the paper we use $A$ for the area of a domain, 
$L$ for its  perimeter, $R$ for its inradius (the supremum of the radii of disks contained in the domain) and $d$ for its diameter.

The problem of finding good bounds for eigenvalues of various domains has been of interest for many years. See for example \cites{PS,Du,AB,Pr} for general bounds, \cites{S2,AF} for result about polygons,
and \cites{F,F2,S} for bounds for triangles. Comprehensive overview of this subject along with methods used to tackle it can be found in the book \cite{H}. 
The results contained in \cite{F2} are especially interesting, since they are of different nature than our bounds and it should be possible to combine methods from this paper with our approach to get even better bounds. 

In \cite{PS}*{Section 7.4}, P\'olya and Szeg\"o conjectured that with fixed area, amongst all polygons with $n$ sides the regular one
minimizes the first eigenvalue (they also conjectured that the inner radius and the transfinite diameter are also minimized, but these has been proved). 
The conjecture remains open except for the following cases: 
\begin{gather}
  \lambda_1 A|_{Triangle}\geq \lambda_1 A|_{Equilateral},\label{tA}\\
  \lambda_1 A|_{Quadrilateral}\geq \lambda_1 A|_{Square},\label{qA}\\
  \lambda_1 A\geq\lambda_1 A|_{Ball}.\label{bA}
\end{gather}
The notation $|_{set}$ is used to indicate the set for the quantity to the left of it.
In the last bound (called the Faber-Krahn inequality) no set is specified since the bound is true for an arbitrary domain. The proofs of these results along with the
conjecture for polygons can be found in \cite{H}. Note that the ball in the last bound can be viewed as a limiting case of a regular polygon with an infinite number of sides.

There are also upper bounds where in this case the ball is the extremal.
\begin{gather}
  \lambda_1R^2\leq \lambda_1R^2|_{Ball}\label{bR},\\
  {\lambda_2\over\lambda_1}\leq \left.{\lambda_2\over\lambda_1}\right|_{Ball},
\end{gather}
\begin{gather}
  (\lambda_2-\lambda_1)R^2\leq (\lambda_2-\lambda_1)R^2|_{Ball}.
\end{gather}
The first inequality follows trivially from the domain monotonicity. The second and the third
are known as the Payne-P\'olya-Weinberger conjecture proved by Ashbaugh and Benguria \cite{AB}.
Note that there is no scaling factor (area or inradius) in the second bound.
The difference of the eigenvalues in the last bound is called the spectral gap and
it is important in the study of dynamical systems. It can be regarded as a measure of the speed of convergence to equilibrium.  For more on this, see \cite{bame} and \cite{smits}.

We can see that a ball gives both upper and lower bounds for general domains.
We want to show that an equilateral triangle has exactly the same properties among 
triangles. We already have the lower bound (\ref{tA}) as an analog of (\ref{bA}).
It is proved in \cite{S} that
\begin{gather}  
  \lambda_1R^2|_{Triangle}\leq\lambda_1R^2|_{Equilateral}.\label{tR}
\end{gather}
This is the parallel of the bound (\ref{bR}). 
The first  result of this paper is the following
\begin{theorem}\label{upper}
  For arbitrary triangle
  \begin{gather}\label{up1}
    (\lambda_2-\lambda_1)R^2|_{Triangle}\leq (\lambda_2-\lambda_1)R^2|_{Equilateral}.
  \end{gather}
  If a triangle is acute we also have
  \begin{gather}\label{up2}
    \left.{\lambda_2\over\lambda_1}\right|_{Triangle}\leq \left.{\lambda_2\over\lambda_1}\right|_{Equilateral}.
  \end{gather}
\end{theorem}
The additional assumption in the second part of the theorem is due to the method used to
prove the result. We believe this result should hold for all triangles.   The proof of the theorem relies on a variational formula for eigenvalues, as well as some very cumbersome computations.

In view of these results, we venture to propose the following generalization of 
P\'olya's conjecture about polygons
\begin{conjecture}
Let $P(n)$ denote a polygon with $n$ sides and $R(n)$ a regular polygon with
$n$ sides. Then
\begin{gather}
  \lambda_1 A|_{P(n)}\geq \lambda_1 A|_{R(n)},\label{pA}\\
  \lambda_1R^2|_{P(n)}\leq\lambda_1R^2|_{R(n)},\label{cR}\\
  (\lambda_2-\lambda_1)R^2|_{P(n)}\leq (\lambda_2-\lambda_1)R^2|_{R(n)},\\
  \left.{\lambda_2\over\lambda_1}\right|_{P(n)}\leq \left.{\lambda_2\over\lambda_1}\right|_{R(n)}.
\end{gather}
\end{conjecture}
It is easy to check that this conjecture is true for rectangles and that
for quadrilaterals (\ref{pA}) is the same as (\ref{qA}). Our results along with (\ref{tA}) and (\ref{tR}) prove this conjecture for triangles, except for obtuse triangles in the last bound.
All other cases remain open. It is worth noting that a slightly weaker version of (\ref{cR}) is a part of \cite{S2}*{Theorem 2}. The only difference is that the scaling factor $R$ is missing and the regular polygon has outer radius $1$ instead of inradius $1$. As we have been recently informed by the author of \cite{S2}, the proof of this theorem should work if we replace outer radius $1$ with inradius $1$. This would prove inequality (\ref{cR}).

The second goal of this paper is to establish sharper lower bounds for the first
eigenvalues of arbitrary triangles. In addition to (\ref{tA}) we have
\begin{gather}
  \lambda_1\geq \frac{\pi^2}4\left( R^{-2}+d^{-2}\right),\label{protter}\\
  \lambda_1\geq \pi^2\left( \frac4{d^2}+\frac{d^2}{4A^2} \right)\label{freitas}.
\end{gather}
The first bound is due to Protter \cite{Pr} and the second result was proved recently by
Freitas \cite{F}. Each of these  bounds is the better than the other bounds for certain triangles, but
not as good for some others. 

We obtained new lower bounds for triangles which are better than (\ref{protter}) and (\ref{freitas})
whenever these are better than (\ref{tA}).
\begin{theorem}\label{lower1}
  For an arbitrary triangle with area $A$, diameter $d$ and shortest altitude $h$,
  \begin{gather}
    \lambda_1|_{Triangle}\geq \pi^2\left( \frac4{d^2+h^2}+\frac{d^2+h^2}{4A^2} \right).
  \end{gather}
\end{theorem}
We also obtained a sharp bound based on circular sectors.
\begin{theorem}\label{lower2}
  Let $\gamma$ be the smallest angle of a triangle.
  Denote by $I(A,\gamma)$ an isosceles triangle with same area $A$ and the
  vertex angle $\gamma$, and $S(A,\gamma)$ a circular sector with area $A$ and 
  angle $\gamma$. Then
  \begin{gather}
    \lambda_1|_{Triangle}\geq \lambda_1|_{I(A,\gamma)}\geq \lambda_1|_{S(A,\gamma)}.
  \end{gather}
  The function
  \begin{gather}
    f(\gamma)=\lambda_1|_{I(A,\gamma)}
  \end{gather}
  is decreasing for $\gamma\in(0,\pi/3)$ and increasing for $\gamma\in(\pi/3,\pi)$.
\end{theorem}

This last result can be viewed as a generalization of the P\'olya's isoperimetric 
bound (\ref{tA}). If we fix $A$ and the smallest angle, then the isosceles triangle
minimizes the first eigenvalue. Then, due to the monotonicity property, we also get an alternative
prove of (\ref{tA}). 
The bounds involving circular sectors can be used to get good lower bounds for
the eigenvalues since the eigenvalues of sectors are given explicitly in terms
of the zeros of the Bessel function. 

From now on $I(A,\gamma)$ will always denote an isosceles triangle with area $A$ and angle $\gamma$ between its sides of equal length. This angle will be called the vertex angle. The side opposite to that angle will be called the base and the other two the arms.  

The methods used to prove (\ref{freitas}) and (\ref{protter}) are not based on any kind of symmetrization argument.
In contrast, the proofs of our lower bounds rely on certain symmetrization techniques similar
to those used in the proof of P\'olya's isoperimetric inequality. In this sense, our results can be viewed as  generalized isoperimetric bounds. Using the same
techniques we also give an alternative proof of Freitas's bound (\ref{freitas}). This shows that symmetrization
is the best way for obtaining lower bounds for the eigenvalues, at least for triangles.
The bound is determined by the shape of the symmetrized domains.  That is, the equilateral triangle
in (\ref{tA}), the rectangles in Theorem \ref{lower1} and (\ref{freitas}), or 
the circular sector in Theorem \ref{lower2}. 

The rest of the paper is organized as follows. In Section 2 we compare the new lower
bounds with the known results. Section 3 contains the definitions and basic explanations of various forms of symmetrization techniques which are used in Section 4 to prove the lower bounds.
Section 5 contains the proofs of the upper bounds. Variational methods are used to
obtain complicated polynomial bounds for the eigenvalues. These bounds are then
simplified to the final results by proving certain polynomial inequalities.
The algorithm for solving such polynomial inequalities is given in  Section 6. The last section contains a script in Mathematica used to perform long calculations.

\section{Comparison of lower bounds}\label{scomp}

In this section we show that our bound (Theorem \ref{lower1}) is sharper than the
bound obtained by Freitas (\ref{freitas}) whenever the latter is better than
P\'olya's isoperimetric bound (\ref{tA}). We also give a numerical comparison between
the lower bounds to show that Theorem \ref{lower2} in practice gives the best lower bounds
for a wide class of triangles.

Both bounds (\ref{freitas}) and Theorem \ref{lower1} are of the form
\begin{gather}\label{genform}
  f(x)=\pi^2\left( \frac4x+\frac{x}{4A^2} \right).
\end{gather}
If we write the explicit value for the eigenvalue of the equilateral triangle,
the bound (\ref{tA}) reads
\begin{gather}
  \lambda_1\geq \frac{4\sqrt3 \pi^2}{3A}.
\end{gather}
To compare it to the other two we need to investigate the following inequality
\begin{gather}
  \frac{4\sqrt3}{3A}\geq \frac4x+\frac{x}{4A^2}.
\end{gather}
Put $x=4yA$, then
\begin{gather}
  \frac{4\sqrt3}3\geq \frac1y+\frac{y}1.
\end{gather}
One can check that the equality holds if $y=\sqrt3$ or $y=\sqrt3/3$. 
Hence, P\'olya's bound (\ref{tA}) is better than a bound of the type (\ref{genform})
if $x\in\left( 4A\sqrt3/3,4A\sqrt3 \right)$. We also observe that $f(x)$ is
increasing for $x\geq 4A\sqrt3$.

In the case of Freitas's bound (\ref{freitas}) (proved in \cite{F}) we get $x=d^2$. Hence this bound is better than P\'olya's bound if
\begin{gather}
  d^2\not\in\left( 4A\sqrt3/3,4A\sqrt3 \right).
\end{gather}
If we denote the length of the altitude perpendicular to the side of length $d$ by $h$ we see that
\begin{gather}
  d\not\in\left( 2h\sqrt3/3,2h\sqrt3 \right).
\end{gather}
Observe that $d$ cannot be smaller than $2h\sqrt3/3$ and it is equal to this
quantity only for an equilateral triangle. Hence Freitas's bound is better than P\'olya's bound
if
\begin{gather}\label{frbetter}
  d>2h\sqrt3.
\end{gather}

In Theorem \ref{lower1} we have $x=d^2+h^2$ which is a bigger value than in Freitas's bound.
Therefore this bound is the best for every triangle such that the above condition is true. 
Figure \ref{fig1} shows where each of the lower bounds (Protter (\ref{protter}), P\'olya (\ref{tA}) and Theorem \ref{lower1}) is the best.
\begin{figure}[htb]
\centering
\beginpgfgraphicnamed{triangles_pic1}
\begin{tikzpicture}[scale=5]
  \filldraw[fill=black!20,draw=black] plot coordinates {(1,1)(.9,1)(.89355,.98077)(.85115,.88635)(.80927,.79051)(.76688,.69609)(.7259,.60116)(.68351,.50674)(.64112,.41232)(.59873,.31791)(.55633,.22349)(.51394,.12907)(.47013,.03516)(.46384,.01923)(.455,0)(1,0)};
  \filldraw[fill=black!35,draw=black] plot coordinates {(0,.7)(.01923,.68113)(.05488,.63778)(.08667,.59057)(.11654,.54143)(.14352,.4894)(.16953,.43641)(.19362,.38149)(.21674,.32561)(.23794,.2678)(.25721,.20807)(.27648,.14834)(.29189,.08475)(.30731,.02116)(.31,0)(0,0)};
  \draw (0,1) -- (1,1);
  \draw[<->] (0,1.05) -- (0,0) -- (1.05,0);
  \foreach \x in {1,3,5,7}{
  \draw (\x/7-1/7,0.01) -- (\x/7-1/7,-0.01) node[below] {\small \x};}
  \foreach \x in {0,0.5,1}{
  \draw (0.01,\x) -- (-0.01,\x) node[left] {\small \x};}
  \draw (1/7,0) node[below] {\small $M$};
  \draw (0,0.25) node[left] {\small $U$};
  \draw (0.14,0.1) node {\small P\'olya};
  \draw (0.4,0.8) node {\small Theorem \ref{lower1}};
  \draw (0.8,0.3) node {\small Protter};
\end{tikzpicture}
\endpgfgraphicnamed

\caption{Theorem \ref{lower1}}
  \label{fig1}
\end{figure}

On this, and all other figures, $M$ denotes the side with the middle length, and $U$ is equal to the difference between the length of the longest side and $M$. The shortest side is assumed to be $1$, hence $U\leq 1$.
This gives a one-to-one mapping of all the triangles onto the infinite strip $[0,1)\times[1,\infty)$.

Next, we give some numerical results with Theorem \ref{lower2} included. Since the bound involving the eigenvalues of circular sectors rely on the calculations of the zeros of the Bessel function, it is hard to compare to the other bounds. The numerical comparisons from Figure \ref{fig2} show that this bound is better than the other bounds for almost all triangles. 

One could also use a simplified sector based lower bound by just taking the smallest sector containing a given triangle. This gives the  sector $S(\gamma N^2/2,\gamma)$ with the same angle as in Theorem \ref{lower2}, but larger radius. This is clearly worse than Theorem \ref{lower2}, but still gives a good bound as can be seen on Figure \ref{fig3}.

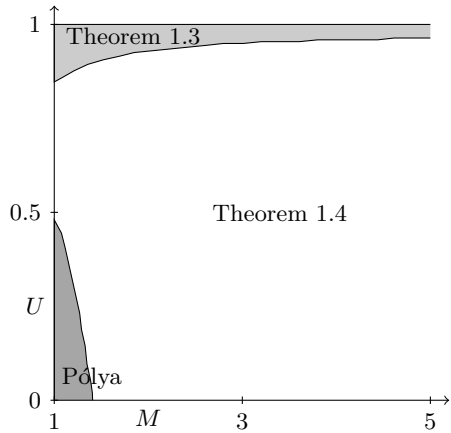
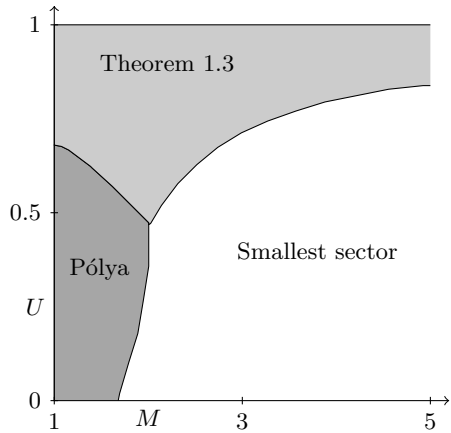
\begin{figure}[htb]
  \centering
  \subfloat[Theorem \ref{lower2}]{
\beginpgfgraphicnamed{triangles_pic2}
\begin{tikzpicture}[scale=5]
   \filldraw[fill=black!20,draw=black] plot coordinates {(1,1)(0,1)(0,.847)(.01923,.85756)(.05305,.87688)(.08929,.8938)(.13036,.90588)(.17385,.91554)(.2125,.9252)(.26082,.93003)(.30914,.93487)(.35746,.9397)(.40578,.94453)(.44927,.94936)(.50242,.94936)(.55073,.95419)(.60388,.95419)(.6522,.95419)(.70052,.95903)(.75367,.95903)(.80682,.95903)(.85997,.95903)(.90346,.96386)(.95661,.96386)(.98077,.96386)(1,.96386)};
    \filldraw[fill=black!35,draw=black] plot coordinates {(0,0)(0,.48)(.01923,.44443)(.02889,.40578)(.03856,.36229)(.04822,.31881)(.05789,.27532)(.06755,.23183)(.07238,.18835)(.08204,.14486)(.08688,.09654)(.09654,.05305)(.10137,.01923)(.102,.0)};
  \draw[<->] (0,1.05) -- (0,0) -- (1.05,0);
  \foreach \x in {1,3,5}{
  \draw (\x/4-1/4,0.01) -- (\x/4-1/4,-0.01) node[below] {\small \x};}
  \foreach \x in {0,0.5,1}{
  \draw (0.01,\x) -- (-0.01,\x) node[left] {\small \x};}
  \draw (1/4,0) node[below] {\small $M$};
  \draw (0,0.25) node[left] {\small $U$};
  \draw (0.1,0.06) node {\small P\'olya};
  \draw (0.6,0.5) node {\small Theorem \ref{lower2}};
  \draw (0.21,0.97) node {\small Theorem \ref{lower1}};
  \end{tikzpicture}
\endpgfgraphicnamed

  \label{fig2}
  }
  \subfloat[Smallest sector containing a triangle]{
\beginpgfgraphicnamed{triangles_pic3}
\begin{tikzpicture}[scale=5]
    \filldraw[fill=black!35,draw=black] plot coordinates {(0,0)(0,.68)(.01923,.67636)(.03856,.6667)(.09654,.62321)(.10137,.61838)(.15452,.57006)(.20284,.52174)(.25116,.47342)(.25116,.35746)(.23666,.26566)(.22217,.17868)(.19801,.10137)(.17385,.01923)(.17,0)};
    \filldraw[fill=black!20,draw=black] plot coordinates {(1,.83823)(.98077,.83823)(.88896,.82857)(.80441,.81165)(.71985,.79474)(.64254,.77058)(.56765,.74401)(.49758,.7126)(.43477,.67395)(.3792,.62804)(.32847,.57731)(.28498,.51933)(.25599,.47101)(.25116,.46859)(.25116,.47342)(.20284,.52174)(.15452,.57006)(.10137,.61838)(.09654,.62321)(.03856,.6667)(.01923,.67636)(0,.68)(0,1)(1,1)};
  \draw[<->] (0,1.05) -- (0,0) -- (1.05,0);
  \foreach \x in {1,3,5}{
  \draw (\x/4-1/4,0.01) -- (\x/4-1/4,-0.01) node[below] {\small \x};}
  \foreach \x in {0,0.5,1}{
  \draw (0.01,\x) -- (-0.01,\x) node[left] {\small \x};}
  \draw (1/4,0) node[below] {\small $M$};
  \draw (0,0.25) node[left] {\small $U$};
  \draw (0.12,0.35) node {\small P\'olya};
  \draw (0.7,0.4) node {\small Smallest sector};
  \draw (0.3,0.9) node {\small Theorem \ref{lower1}};
  \end{tikzpicture}
\endpgfgraphicnamed

  \label{fig3}
}
\caption{Sector-based bounds}
\end{figure}

We do note here that there is no best bound, although our
new bounds together with P\'olya's give the best results depending on the triangle. We also see that symmetrization techniques (all bounds are consequences of certain symmetrization procedures) are very powerful and lead to very good lower bounds for the first eigenvalue of triangles.

It would be also interesting to see how far the bounds are from the exact values. We begin with a right isosceles triangle with arms of length $1$. The first eigenvalue is known in this case and it is equal to $5\pi^2$. The lower bounds results are given in Table \ref{tab1}.
Clearly, neither bound is particularly accurate, although P\'olya's bound and Theorem \ref{lower2} are the closest.

The latter bound works best for ``taller'' triangles. Consider a right triangle with angle $\pi/6$ and hypotenuse $2$. Table \ref{tab2} shows the
values of the lower bounds for this triangle.
\begin{table}[htb]
  \centering
  \subfloat[Right isosceles triangle]{
  \begin{tabular}{|l|l|}
    \hline
  exact& 49.348\\
  \hline
  P\'olya (\ref{tA})& 45.5858\\
  Freitas (\ref{freitas})& 39.4784\\
  Protter (\ref{protter})& 29.9958\\
  Theorem \ref{lower1}& 40.4654\\
  Theorem \ref{lower2}& 45.2255\\
  \hline
\end{tabular}
  \label{tab1}
  }
  \hspace{1cm}
  \subfloat[Half of the equilateral triangle]{
  \begin{tabular}{|l|l|}
    \hline
  exact& 30.7054\\
  \hline
  P\'olya (\ref{tA})& 26.3189\\
  Freitas (\ref{freitas})& 23.0291\\
  Protter (\ref{protter})& 19.0338\\
  Theorem \ref{lower1}& 23.9381\\
  Theorem \ref{lower2}& 29.8449\\
  \hline
\end{tabular}
  \label{tab2}
}
\caption{Triangles with known first eigenvalues}
\end{table}
This time Theorem \ref{lower2} gives a very close bound, while the others are not so accurate. 
The advantage of the sector bound is the biggest for very ``tall'' triangles. For an isosceles triangle with base $1$ and fixed arms we get Table \ref{tab3}.
The exact values are not known, but the difference between the bounds given by Theorem \ref{lower2} and the other bounds is clear.

\begin{table}[htb]
  \centering
  \begin{tabular}{|l|l|l|}
    \hline
   & arm $2$ & arm $4$\\
   \hline
   P\'olya (\ref{tA})&23.5404&11.4865\\
   Freitas (\ref{freitas})&20.3972&12.4937\\
   Protter (\ref{protter})&17.0662&12.8437\\
   Theorem \ref{lower1}&20.9906&12.9675\\
   Theorem \ref{lower2}&27.0781&18.8754\\
   Upper bound (6.1) in \cite{S}&27.6695&18.9749\\
   \hline
 \end{tabular}
  \caption{Tall isosceles triangles}
  \label{tab3}
 \end{table}

 All these numerical results show that the lower bounds can be very accurate if the triangle is acute and almost isosceles (P\'olya's bound or Theorem \ref{lower2}). Unfortunately, neither bound is very good in the case of ``wide'' obtuse triangles. The values of the lower bounds for an isosceles triangle with base $1.95$ and arms $1$ are given in Table \ref{tab4}.

\begin{table}[htb]
  \centering
 \begin{tabular}{|l|l|}
   \hline
   P\'olya (\ref{tA})&105.206\\
   Freitas (\ref{freitas})&210.273\\
   Protter (\ref{protter})&205.698\\
   Theorem \ref{lower1}&212.735\\
   Theorem \ref{lower2}&185.161\\
   \hline
   Conjecture 1.2 in \cite{S}&\\
   \hline
   Lower bound&251.077\\
   Upper bound&299.7\\
   \hline
 \end{tabular}
  \caption{Wide isosceles triangle}
  \label{tab4}
\end{table}

Clearly Freitas, Protter and Theorem \ref{lower1} are the best in this case, but those values are not very accurate. In fact, neither bound allows us to prove the second part of the Theorem \ref{upper} in the case of obtuse triangles. Note that Conjecture 1.2 in \cite{S} could provide a lower bound strong enough for this task. Theorems 2 and 5.1 in \cite{F2} could also provide good enough bounds.

\section{Symmetrization techniques}\label{ssym}
In this section we present various geometric transformations that decrease the first eigenvalue. The following subsections contain strict definitions of three different kinds of symmetrization.
Here we just remark that the important common property of those transformations is that they are contractions on $W^{1,2}(\Rd)\cap C_c(\Rd)$ and isometries on $L^2(\Rd)$. 
Those properties, together with the minimax formula for the first Dirichlet eigenvalue, give $\lambda_1(\Omega)\geq \lambda_1(\Omega^*)$. The most general reference for those results is \cite{H}.

\subsection{Steiner Symmetrization}
We start with the well known Steiner symmetrization 
(see for example \cite{PS}*{Note A} or \cite{H}*{Chapter 2}).
Fix a line $l=\{ax+b:x\in\Real\}$ where $a$ and $b$ are arbitrary points on the plane.
Let $\{l_t\}_{t\in\Real}$ be a family of lines perpendicular to $l$ such that for each $t$ the line $l_t$ passes through the point $at+b$.
For an arbitrary domain $\Omega$ we define its Steiner symmetrization $\Omega^*$ with respect to $l$ 
as a convex domain symmetric with respect to $l$ and such that for every line $l_t$ 
\begin{gather}
  |\Omega^*\cap l_t|=|\Omega\cap l_t|,
\end{gather}
where $|\cdot|$ denotes the $1$-dimensional Lebesgue measure.

Intuitively, we look at the cross-sections of $\Omega$ perpendicular to $l$ and we
center them around $l$ (see Figure \ref{sym} for an example).

This procedure has many interesting properties (see \cites{H,PS}).
While the area remains fixed, the perimeter decreases and the inradius increases.
But the most interesting property from our point of view is that the first
eigenvalue of the Dirichlet Laplacian on $\Omega$ is bigger than that on $\Omega^*$.
This basic property has been widely used in the proofs of many isoperimetric bounds
for eigenvalues in many settings, see for example \cites{AB,PS,H}. 
We will use this property to prove our results.

\subsection{Continuous Steiner Symmetrization}
The second type of symmetrization is the continuous Steiner symmetrization introduced
by P\'olya and Szeg\"o in \cite{PS}*{Note B}. Different versions of it have been studied by Solynin \cite{S1} and Brock \cites{B1,B2}.  The difference in case of convex domains is only in the way of defining ``continuity parameters''. The version studied in \cites{B1,B2} is more general and it works even for not connected domains. We refer the reader to these three papers and to \cite{H} for the
properties of this transformation. Here we only give the definition valid for convex domains.
As above we look at all the intervals $(a_t,b_t)$ which are the intersections of $\Omega$ with $l_t$.
Let $(a_t',b_t')$ be the Steiner symmetrized interval.
Let $a_t^\alpha = a_t+\alpha(a_t'-a_t)$ and $b_t^\alpha = b_t+\alpha(b_t'-b_t)$. 
Hence we are shifting the intersections with constant speed from their initial
position to the fully symmetrized position.
We define the continuous Steiner symmetrization $\Omega^\alpha$ of a domain $\Omega$ by
\begin{gather}
\Omega^\alpha=\bigcup_{t\in\Real}(a_t^\alpha,b_t^\alpha).
\end{gather}

\begin{figure}[htb]
  \begin{center}
\beginpgfgraphicnamed{triangles_pic4}
\begin{tikzpicture}[scale=1.4]
\fill[black!25] (0.185059,0.257354) -- (2.43598,2.73219) -- (8.18506,0.257354);
\draw[dotted,-stealth'] (0.577486,2.73219) -- (4.18506,2.73219);
\draw (4.18506,2.73219) -- (0.185059,0.257354)
  (4.18506,2.73219) -- (8.18506,0.257354)
  (8.18506,0.257354) -- (0.577486,2.73219)
  (0.185059,0.257354) -- (8.18506,0.257354)
  (0.577486,2.73219) -- (0.185059,0.257354)
  (4.18506,-0.2) -- (4.18506,3.5);
\fill
  (0.577486,2.73219) circle (0.6pt)
  (8.18506,0.257354) circle (0.6pt)
  (0.185059,0.257354) circle (0.6pt)
  (2.43598,2.73219) circle (0.6pt)
  (4.18506,2.73219) circle (0.6pt);
\draw (0.577486,2.73219) node[above] {\small $\Omega^0=\Omega$}
  (2.43598,2.73219) node[above] {\small $\Omega^\alpha$ }
%  (2.53383,0.157354) node[below] {\small $d$ }
  (4.18506,2.73219) node[above right] {\small $\Omega^1=\Omega^*$ }
  (4.25506,3.15) node[left] {\small $l$ };
  \end{tikzpicture}
\endpgfgraphicnamed

  \end{center}
\caption{Continuous Steiner symmetrization}
\label{sym}
\end{figure}
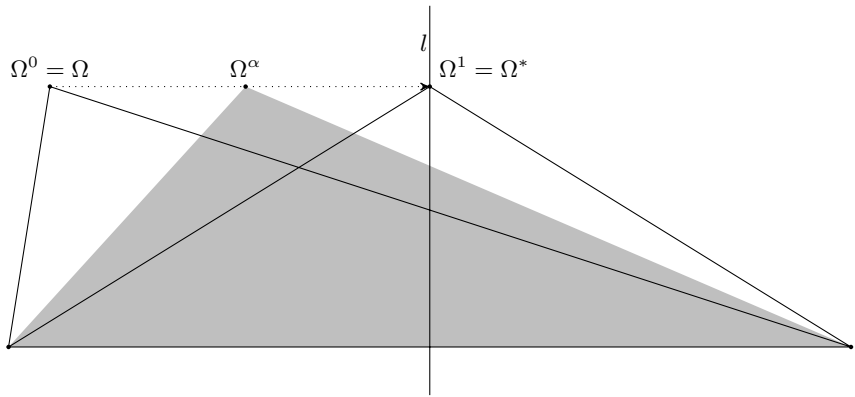

Figure \ref{sym} shows the action of continuous Steiner symmetrization on triangles. We can see that $\Omega^0=\Omega$ and $\Omega^1=\Omega^*$.
The first eigenvalue of the domain $\Omega^\alpha$ is decreasing when
$\alpha$ is increasing (see \cites{B1,B2,H}) just like in the case of the classical 
Steiner symmetrization. Note that our ``continuity parameter'' $\alpha$ is related to the one in \cite{H} by $1-\alpha=e^{-t}$.

\subsection{Polarization}
The last technique we want to introduce is called  polarization. It was used in \cite{Du} to prove general inequalities for capacities and eigenvalues.
It has also been useful in studying other types of symmetrization (see \cites{S1,BS}) and heat kernels for certain operators (see \cite{Dr}). 
Let $l$ be an arbitrary line, $H_1$ and $H_2$ two halfspaces with boundary $l$.
For $x\in H_2$ let $\bar{x}$ denote the reflection of $x$ with respect to $l$.
Polarization of a domain $\Omega$ is defined pointwise by the following transformation
\begin{definition}
  The polarization of a set $\Omega$ with respect to $l$ is a set $\Omega^P$ with the following properties.
  \begin{enumerate}
    \item If $x\in \Omega\cap H_1$ then $x\in \Omega^P$.
    \item If $x\in \Omega\cap H_2$ then $\bar{x}\in \Omega^P$.
    \item If both $x$ and $\bar{x}$ are in $\Omega$, then both $x$ and $\bar{x}$ are in $\Omega^P$,
  \end{enumerate}
\end{definition}

In essence, we reflect the set $\Omega$ with respect to $l$, the intersection of $\Omega$
with its reflection is in the polarized domain $\Omega^P$ and all other points of $H_1$ that
are in either $\Omega$ or its reflection also belong to $\Omega^P$. Figure \ref{polar} shows a triangle,
its reflection with respect to $l$ and the polarized domain (outlined polygon).

\begin{figure}[htb] 
  \centering
\beginpgfgraphicnamed{triangles_pic5}
\begin{tikzpicture}
\fill[black!25] (0.184039,2.02254) -- (2.87248,0.127085) -- (7.01153,4.3435);
\fill[black!45] (0.184039,3.67177) -- (2.87248,5.56722) -- (7.01153,1.35081);
\draw (4.80881,2.09961) -- (0.184039,3.67177) -- (2.87248,5.56722) -- (4.80881,3.5947) -- (7.01153,4.3435) -- (4.80881,2.09961)
(0,2.84715) -- (7.31157,2.84715) node[very near start,below] {\scriptsize $l$};
  \end{tikzpicture}
\endpgfgraphicnamed

  \caption{Polarization of sets}
\label{polar}
\end{figure}
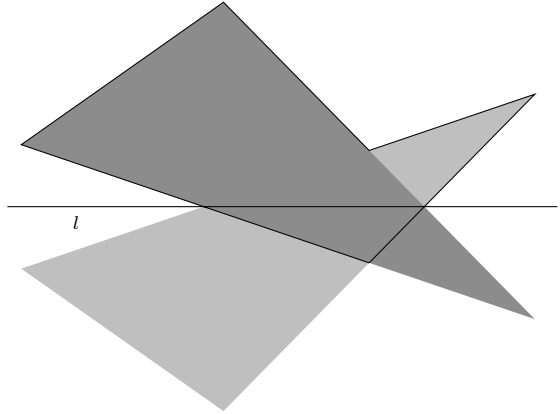

This simple procedure has many interesting monotonicity properties, among them is 
the monotonicity of the eigenvalues. For this, and many other properties we
refer the reader to \cites{BS,Dr} and to references in these papers.
\section{The proofs of the lower bounds}\label{splow}

Before we prove Theorems \ref{lower1} and \ref{lower2} let us present a one simple application of the continuous Steiner symmetrization. This and other kinds of symmetrization were introduced in Section \ref{ssym}.

Let $T$ be an acute isosceles triangle with area $A$ and arms of length $d$. Let $l$
be a line perpendicular to one of the arms and passing through the midpoint of this arm.
If we apply the continuous Steiner symmetrization with respect to $l$ we get
\begin{lemma}
  Amongst triangles $T$ with diameter $d$ and area $A$, isosceles triangle with arms $d$
  maximizes the first eigenvalue and isosceles triangle with base $d$ minimizes
  it (See Figure \ref{sym}).
\end{lemma}

This simple result shows that upper bounds for the eigenvalues can also be obtained
using symmetrization techniques. Using the largest circular sector contained in $T$ we can get very accurate upper bounds for the eigenvalues of triangles. 

We can use polarization to give an alternate proof of (\ref{freitas}). A very similar procedure, although more complicated, will be used later to prove Theorem \ref{lower2}.
Let $T$ be an arbitrary triangle with the longest side of length $d$. 
First, we use the Steiner symmetrization with
respect to a line perpendicular to the longest side. As a result we have an
isosceles triangle $ABC$ with the base of length $d$.

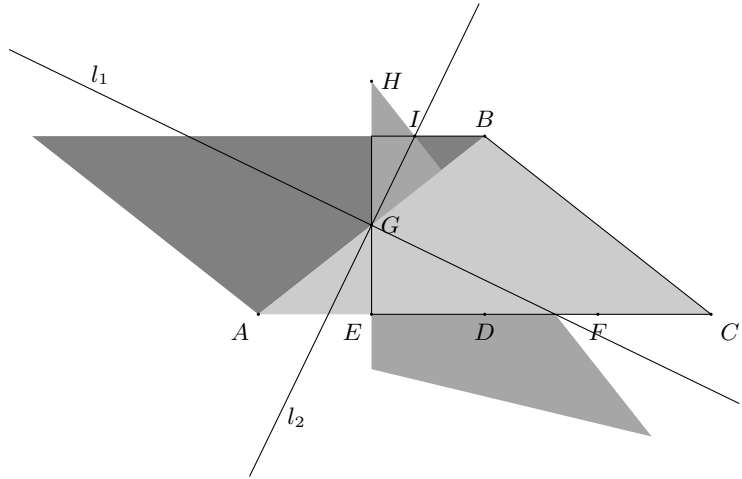
\begin{figure}[htb]
  \centering
\beginpgfgraphicnamed{triangles_pic6}
\begin{tikzpicture}[scale=0.06]
\fill[black!50] (105.404,75.4777) -- node[at start,above,black] {\small $B$} (5.10483,75.4777) -- (55.2546,35.9687) node [at end,below left,black] {\small $A$};
\fill[black!35] (80.3295,87.6448) -- node[at start,right,black] {\small $H$} (142.399,8.85811) -- (80.3295,23.8016);
\fill[black!20] (55.2546,35.9687) -- (155.554,35.9687) -- (105.404,75.4777);
\draw (0,94.7016) -- (161.821,16.1811) node[very near start,above] {\small $l_1$}
(104.181,104.878) -- (53.2908,0) node[very near end,right] {\small $l_2$}
(80.3295,35.9687) -- (80.3295,75.4777) -- (105.404,75.4777) -- (155.554,35.9687) -- cycle;
\fill (55.2546,35.9687) circle (10pt)
  (80.3295,35.9687) circle (10pt)
  (105.404,35.9687) circle (10pt)
  (130.479,35.9687) circle (10pt)
  (155.554,35.9687) circle (10pt)
  (80.3295,55.7232) circle (10pt)
  (89.915,75.4777) circle (10pt)
  (105.404,75.4777) circle (10pt)
  (80.3295,87.6448) circle (10pt);
\draw (80.3295,35.9687) node[below left] {\small $E$}
(105.404,35.9687) node[below] {\small $D$}
(130.479,35.9687) node[below] {\small $F$}
(155.554,35.9687) node[below right] {\small $C$}
(80.3295,55.7232) node[right] {\small $G$}
(89.915,75.4777) node[above] {\small $I$};
  \end{tikzpicture}
\endpgfgraphicnamed

  \caption{Freitas's bound}
  \label{figfr}
\end{figure}

The further construction is shown on Figure \ref{figfr}.
We divide the longest side $AC$ into four equal parts. Then we construct the line
$EH$ perpendicular to $AC$. Finally we construct the bisector $l_1$ of the angle $AGH$.
We apply the polarization with respect to the line $l_1$. The darker triangle with
vertex $H$ is the result of the reflection with respect to $l_1$. Note that we need
the line $l_1$ to cut the side $AC$ between the points $E$ and $F$. If not, then
we would not be able to symmetrize the other half of the triangle (the bisector would cut
the sector $EG$, giving unnecessary reflections). Fortunately one can check that under
the condition (\ref{frbetter}) this is always true. 

Now we perform one more polarization with respect to the line $l_2$, which is the bisector 
of the angle $HIB$. As can be seen on Figure \ref{figfr}, the left side of the
triangle has been changed into a rectangle. 

If we repeat the procedure on the other side we get the rectangle with the base $EF$
and the height $BD$. But for a rectangle we have the following explicit formula for the first eigenvalue.
\begin{gather}\label{eigrect}
  \lambda_1|_{Rectangle}=\pi^2\left( a^{-2}+b^{-2} \right),
\end{gather}
where $a$ and $b$ are the lengths of the sides. In our case the lengths are
$|EF|=d/2$ and $|BD|=h$. This gives Freitas's bound (\ref{freitas}).

To get a sharper bound we need to symmetrize the triangle into a rectangle with a
longer base and a shorter height. We start with the same Steiner symmetrization as before,
That is, we symmetrize with respect to the line perpendicular to the longest side. Next, we perform
one more Steiner symmetrization but with respect to the longest side. This gives a
rhombus with diagonals of length $|CD|=h$ and $|EF|=h$ (see Figure \ref{figl1}). 

\begin{figure}[htb]
  \centering
\beginpgfgraphicnamed{triangles_pic7}
\begin{tikzpicture}[scale=0.7]
\filldraw[fill=black!25] (0.276618,2.34093) -- (5.27662,4.34093) -- (10.2766,2.34093) -- (5.27662,0.340928) -- cycle;
\draw (0.276618,2.34093)  -- node[at start,left] {\small $A$} (10.2766,2.34093) node[at end,right] {\small $B$};
\draw (5.27662,4.34093) -- node[at start,above] {\small $D$} (3.89731,0.892652) node [at end,below left] {\small $E$};
\draw (5.27662,4.34093) -- (5.27662,0.340928) node [at end,below] {\small $C$};
\fill (0.276618,2.34093) circle (1pt)
(5.27662,4.34093) circle (1pt)
(5.27662,0.340928) circle (1pt)
(10.2766,2.34093) circle (1pt)
(3.89731,0.892652) circle (1pt);
  \end{tikzpicture}
\endpgfgraphicnamed

  \caption{Theorem \ref{lower1}}
  \label{figl1}
\end{figure}

If we apply the Steiner symmetrization one more time but
with respect to the height $DE$ of the rhombus, we obtain the rectangle with base 
$AC$ and height $DE$. Using Pythagorean theorem we find that $|AC|=\sqrt{d^2+h^2}/2$.
Since the area $A$ of the triangle remains constant under symmetrization, we also have 
$|DE|=A/|AC|$. These, together with (\ref{eigrect}), give the proof of Theorem 
\ref{lower1}. Note, that the proof of Theorem \ref{lower1} is actually easier than the proof of Freitas's bound (\ref{freitas}).

It remains to prove Theorem \ref{lower2}. Let us begin with a lemma.
\begin{lemma}\label{symtr}
  Let $T$ be a triangle with the smallest angle $\gamma$. Let $T'$ be a triangle 
  with same area $A$ and same smallest angle $\gamma$ but with a smaller 
  diameter. Then
  \begin{gather}
    \lambda_1|_T\geq \lambda_1|_{T'}.
  \end{gather}
\end{lemma}

\begin{remark}
What this lemma essentially says is that we can continuously deform a triangle into an isosceles triangle, 
preserving the smallest angle and the area in such a way that the first eigenvalue
will be decreasing. This immediately gives the first inequality in Theorem \ref{lower2}.
\end{remark}

\begin{proof}
This is perhaps the most complicated application of any symmetrization techniques. 
We have to apply a suitable sequence of polarizations to obtain the result.
First, let $T_\varepsilon'$ be a triangle similar to $T'$ but with area $A+\varepsilon$.
We will symmetrize $T$ into a set contained in $T_\varepsilon'$. Then
\begin{gather}
  \lambda_1|_T\geq \lambda_1|_{T_\varepsilon'}.
\end{gather}
But when $\varepsilon\to0$, the eigenvalue of $T_\varepsilon'$ converges to the
eigenvalue of $T'$ and this ends the proof. Indeed, the triangles are similar hence we get the convergence due to scaling property of the eigenvalues.

Ideally, we would like to just define a sequence polarizations that folds $T$ inside $T'_\varepsilon$. Unfortunately, it is virtually impossible to choose a correct line even for the first polarization. Each polarization must be with respect to a line cutting the longest and the shortest sides, but its exact position is not clear. The first line should be close to the vertex and the following lines should move away from the vertex. To give a precise position of each line we need to consider a temporary reversed sequence of transformations (very similar to the sequence of polarizations performed in the proof of \ref{freitas}). The first line in this reversed sequence (or the last line for polarization) is easy to define. Having this line we can get another, and so on. The precise construction is split into 4 steps.

\subsubsection*{Step 1: Definition of elementary transformation.}
$\;$

We start with two triangles $T=ABC$ and $T'_\varepsilon=DBE$ such that $E$ is on the interval $BC$ and $A$ is on the interval $DB$. We also assume that the angle with vertex $B$ is the smallest in the triangle $T$ and that the area of $T'_\varepsilon$ equals the area of $T$ plus $\varepsilon$. The triangles are shown on Figure \ref{figl2a}. Let $l_1$ be the bisector of the angle between the intervals $DE$ and $AC$. If we reflect $T$ with respect to this bisector we obtain another triangle $A'B'C'$ shown on Figure \ref{figl2a2}. Let $G$ be the intersection of $C'B'$ and $AC$.

\begin{figure}[htb]
  \begin{center}
    \subfloat[$T$ and $T_\varepsilon'$]{
\beginpgfgraphicnamed{triangles_pic8}
\begin{tikzpicture}[scale=0.04]
\clip (210.94,0) rectangle (351.57,228.52);
%    (89.829,221.742)(249.54,148.072)(178.508,113.04)
%  (206.645,317.241)(247.094,146.072)(295.547,208.722)
%  (369.81,38.5379)(240.348,157.594)(318.697,169.18)
\fill[fill=black!15] (279.188,10.2571) -- (317.958,181.814) -- (246.926,146.782);
\draw (0,122.051) -- (527.36,200.037)
%(281.054,163.614) -- (233.118,317.216)
(279.188,10.2571) -- (314.978,168.63)
%(248.086,147.354) -- (368.55,0)
%(245.488,152.867) -- (0,9.14765)
(244.235,158.169) -- (279.188,10.2571);
\fill (244.235,158.169) circle (10pt)
(246.926,146.782) circle (10pt)
(279.188,10.2571) circle (10pt)
(314.978,168.63) circle (10pt)
(317.958,181.814) circle (10pt);
\draw (244.235,158.169) node[above left] {\small $D$}
(246.926,146.782) node [below left] {\small $A$}
(279.188,10.2571) node [below] {\small $B$}
(314.978,168.63) node [below right] {\small $E$}
(317.958,181.814) node [right] {\small $C$};
%(245,290) node [above left] {\small $l_1$};
    \end{tikzpicture}
\endpgfgraphicnamed

  \label{figl2a}
}
\subfloat[the elementary transformation]{
\beginpgfgraphicnamed{triangles_pic9}
\begin{tikzpicture}[scale=0.04]
\clip (210.94,0) rectangle (351.57,228.52);
%    (89.829,221.742)(249.54,148.072)(178.508,113.04)
%  (206.645,317.241)(247.094,146.072)(295.547,208.722)
\fill[black!30] (369.81,38.5379) -- (240.348,157.594) -- (318.697,169.18);
\fill[black!15] (279.188,10.2571) -- (317.958,181.814) -- (246.926,146.782);
\draw (0,122.051) -- (527.36,200.037);
\draw (281.054,163.614) -- (233.118,317.216) node[near start,left] {\small $l_1$};
\draw (279.188,10.2571) -- (314.978,168.63)
%(248.086,147.354) -- (368.55,0)
%(245.488,152.867) -- (0,9.14765)
(244.235,158.169) -- (279.188,10.2571);
\fill (244.235,158.169) circle (20pt)
(240.348,157.594) circle (20pt)
(246.926,146.782) circle (20pt)
(279.188,10.2571) circle (20pt)
(314.978,168.63) circle (20pt)
(317.958,181.814) circle (20pt)
(250.297,148.445) circle (20pt)
(281.054,163.614) circle (20pt);
\draw (244.235,158.169) node[above right] {\small $D$}
(240.348,157.594) node [above left=-2pt] {\small $C'$}
(246.926,146.782) node [below left] {\small $A$}
(279.188,10.2571) node [below] {\small $B$}
(314.978,168.63) node [below right] {\small $E$}
(317.958,181.814) node [right] {\small $C$}
(250.297,148.445) node [right] {\small $G$}
(281.054,163.614) node [above right] {\small $F$};
\end{tikzpicture}
\endpgfgraphicnamed

  \label{figl2a2}
}

\end{center}
  \caption{Triangles $T=ABC$ and $T_\varepsilon'=DBE$.}
\end{figure}
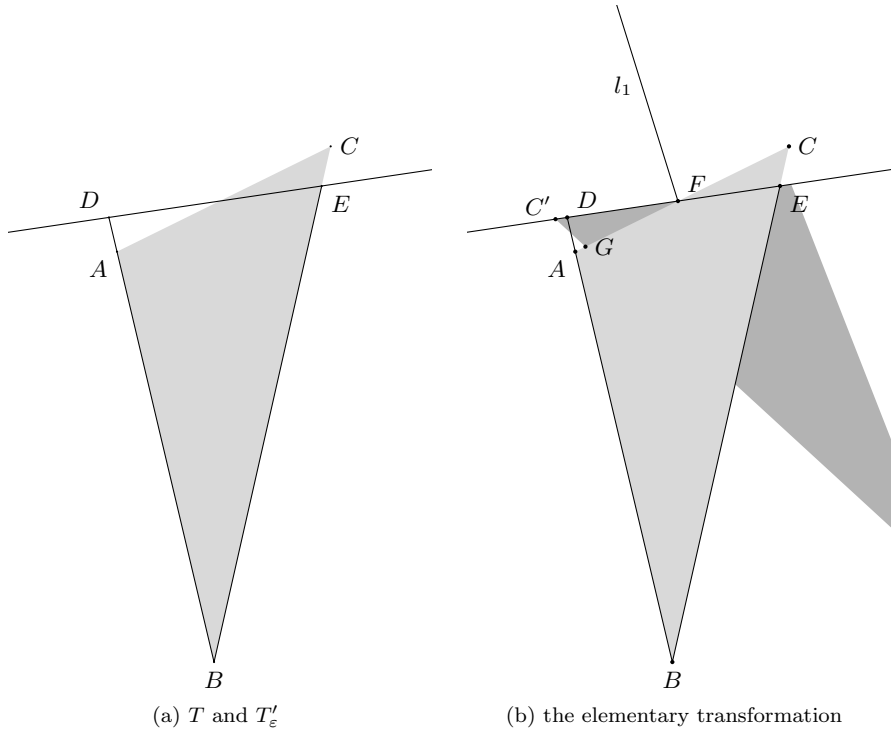

We define the elementary transformation of $T$ to be the polygon $ABEC'G$. This construction is valid due to two conditions. The first, the area of $DBE$ is bigger than the area of $ABC$. The second, the angle $\angle ACB$ is smaller than the angle $\angle EDB$. Those two conditions guarantee that the bisector cuts the sides $AC$ and $BC$ and not the side $AB$ and validates the construction of the elementary transformation.

The same transformation can be applied to more general domains. For example any domain contained in an infinite cone $ACB$ with vertex at $C$. We will need this observation later on.

Notice also that the elementary transformation is equivalent to polarization with respect to the bisector for any domain containing $ABC$ as long as: the intersection of this domain with the bisector is the same as the intersection of $ABC$ with the bisector, the new part (triangle $C'GF$ in the case above) does not intersect the domain. In particular we could take a domain consisting of $ABC$ and any subset of a half-plane disjoint with $ABC$ and with boundary containing $AB$ unless this domain intersects the bisector or $C'GF$. We also do not need $C'GF$ to be a triangle, it is enough that this set does not intersect the domain. 

\subsubsection*{Step 2: Sequence of elementary transformations.}
$\;$

Now we want to describe a sequence of elementary transformations that changes $T=ABC$ into a subset of $T'_\varepsilon=DBE$. 

The first transformation is described in \textit{Step 1}. It is clear that the area of $ADF$ is bigger than the area of $C'GF$. We have
\begin{gather}
  \angle GC'F=\angle ACB<\angle ACB+\angle ABC=\angle DAF.
\end{gather}
Therefore the second elementary transformation can be applied to triangles $C'GF$ and $ADF$ (see Figure \ref{figl2b}). Here we completely disregard the presence of the quadrilateral $AFEB$.

This elementary transformation introduces a new triangular piece that may intersect the triangle $ABC$ near vertex $A$ (see Figure \ref{figl2b2}). The intersection means that the elementary transformation is not equivalent to the polarization of $ABEC'G$ (we did not care about the quadrilateral $AFEB$ while making a transformation). This intersection will not occur if the triangle $ADF$ is large enough to fit the triangle $C'GF$ inside. It would follow that the second elementary transformation is also equivalent to polarization. In such a case the construction of the sequence of transformations would be finished. 

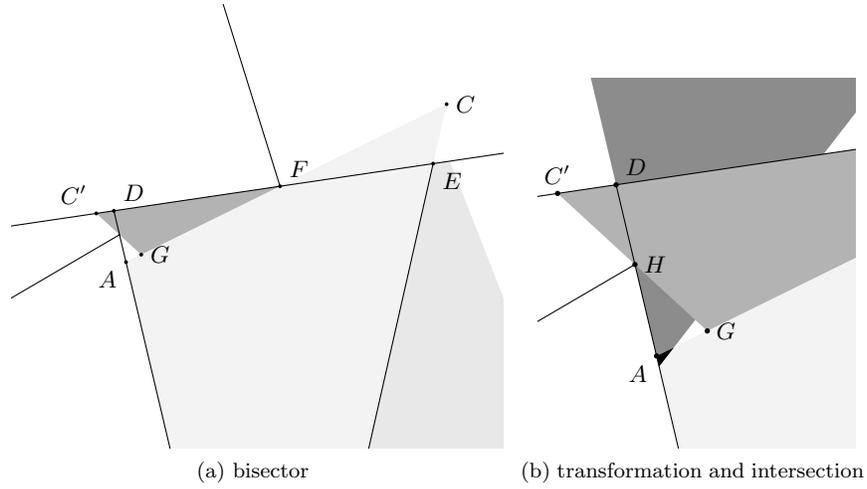
\begin{figure}[htbp]
  \begin{center}
    \subfloat[bisector]{
\beginpgfgraphicnamed{triangles_pic10}
\begin{tikzpicture}[scale=0.06]
\clip (221.49,105.47) rectangle (330.48,203.91);
\fill[black!30] (369.81,38.5379) -- (240.348,157.594) -- (318.697,169.18);
\fill[black!15] (279.188,10.2571) -- (317.958,181.814) -- (246.926,146.782);
\fill[white,fill opacity=0.7]  (246.926,146.782) -- (317.958,181.814) -- (369.81,38.5379) -- (279.188,10.2571);
\draw (0,122.051) -- (527.36,200.037)
(281.054,163.614) -- (233.118,317.216)
(279.188,10.2571) -- (314.978,168.63)
%(248.086,147.354) -- (368.55,0)
(245.488,152.867) -- (0,9.14765)
(244.235,158.169) -- (279.188,10.2571);
\fill (244.235,158.169) circle (12pt)
(246.926,146.782) circle (12pt)
(279.188,10.2571) circle (12pt)
(314.978,168.63) circle (12pt)
(317.958,181.814) circle (12pt)
(281.054,163.614) circle (12pt)
(240.348,157.594) circle (12pt)
(250.297,148.445) circle (12pt);
\draw (244.235,158.169) node[above right] {\small $D$}
(246.926,146.782) node [below left] {\small $A$}
(279.188,10.2571) node [below] {\small $B$}
(314.978,168.63) node [below right] {\small $E$}
(317.958,181.814) node [right] {\small $C$}
(281.054,163.614) node [above right] {\small $F$}
(240.348,157.594) node [above left] {\small $C'$}
(250.297,148.445) node [right] {\small $G$};
    \end{tikzpicture}
\endpgfgraphicnamed

  \label{figl2b}
}
\subfloat[transformation and intersection]{
\beginpgfgraphicnamed{triangles_pic11}
\begin{tikzpicture}[scale=0.2]
\clip (239.07,140.63) rectangle (260.16,165.24);
\fill[black!45] (206.645,317.241) -- (247.094,146.072) -- (295.547,208.722);
\fill[black!30] (369.81,38.5379) -- (240.348,157.594) -- (318.697,169.18);
\fill[black!15] (279.188,10.2571) -- (317.958,181.814) -- (246.926,146.782);
\fill[white,fill opacity=0.7]  (246.926,146.782) -- (317.958,181.814) -- (369.81,38.5379) -- (279.188,10.2571);
 \fill[black] (246.926,146.782) -- (248.086,147.354) -- (247.094,146.072);
\draw (0,122.051) -- (527.36,200.037)
(281.054,163.614) -- (233.118,317.216)
(245.488,152.867) -- (0,9.14765)
(244.235,158.169) -- (279.188,10.2571);
\fill (244.235,158.169) circle (5pt)
(246.926,146.782) circle (5pt)
(279.188,10.2571) circle (5pt)
(314.978,168.63) circle (5pt)
(317.958,181.814) circle (5pt)
(281.054,163.614) circle (5pt)
(240.348,157.594) circle (5pt)
(250.297,148.445) circle (5pt)
(245.488,152.867) circle (5pt);
\draw (244.235,158.169) node[above right] {\small $D$}
(246.926,146.782) node [below left] {\small $A$}
(281.054,163.614) node [above right] {\small $F$}
(240.348,157.594) node [above] {\small $C'$}
(250.297,148.445) node [right] {\small $G$}
(245.488,152.867) node [right] {\small $H$};
\end{tikzpicture}
\endpgfgraphicnamed

  \label{figl2b2}
}
\end{center}
\caption{The second transformation}
\end{figure}

If the intersection occurs, we apply another elementary transformation to the triangle introduced during the previous transformation and to the difference between the triangles considered for the previous transformation (triangle $AGH$). Here we again disregard the presence of other part of the domain. The third elementary transformation is shown on Figure \ref{figl2b3}. To validate this step we need to check that the angle $\angle AGH$ is bigger than the angle $\angle ACB$. Indeed, it is equal to the angle $\angle GC'F$($=\angle ACB$) plus the angle $\angle C'FG$.

\begin{figure}[htb]
  \begin{center}
\beginpgfgraphicnamed{triangles_pic12}
\begin{tikzpicture}[scale=0.2]
\clip (232.04,140.63) rectangle (263.68,165.24);
\fill[black!60] (89.829,221.742) -- (249.54,148.072) -- (178.508,113.04);
\fill[black!45] (206.645,317.241) -- (247.094,146.072) -- (295.547,208.722);
\fill[black!30] (369.81,38.5379) -- (240.348,157.594) -- (318.697,169.18);
\fill[black!15] (279.188,10.2571) -- (317.958,181.814) -- (246.926,146.782);
\fill[white,fill opacity=0.7]  (246.926,146.782) -- (250.297,148.445) -- (240.348,157.594) --  (89.829,221.742) -- (317.958,181.814) -- (369.81,38.5379) -- (279.188,10.2571);
 \fill[black] (246.926,146.782) -- (248.086,147.354) -- (247.094,146.072);
\draw (0,122.051) -- (527.36,200.037)
(281.054,163.614) -- (233.118,317.216)
(248.086,147.354) -- (368.55,0)
(245.488,152.867) -- (0,9.14765)
(244.235,158.169) -- (279.188,10.2571);
\fill (244.235,158.169) circle (5pt)
(246.926,146.782) circle (5pt)
(250.297,148.445) circle (5pt)
(240.348,157.594) circle (5pt)
(245.488,152.867) circle (5pt);
\draw (244.235,158.169) node[above right] {\small $D$}
(246.926,146.782) node [below left] {\small $A$}
(240.348,157.594) node [above] {\small $C'$}
(250.297,148.445) node [right] {\small $G$}
(245.488,152.867) node [right] {\small $H$};
    \end{tikzpicture}
\endpgfgraphicnamed

\end{center}
  \caption{The third transformation}
  \label{figl2b3}
\end{figure}
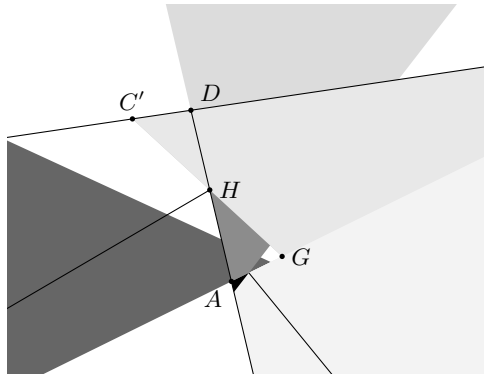

If the newly introduced triangular piece does not intersect the already obtained domain (as on Figure \ref{figl2b3}), the construction in finished. If it does, we apply another elementary transformation (or as many elementary transformations as needed) to fully ``fold'' the triangle $ABC$ inside of the triangle $DBE$. The final subset of $DBE$ that has the same area as $ABC$ will have a spiral-like pattern of triangular pieces introduced by consecutive elementary transformations. First we need to show that this process is finite. Later we need to modify the procedure to get a sequence of valid polarizations. In particular we need to avoid self-intersecting that happens at every stage.

\subsubsection*{Step 3: Finiteness of the construction.}
$\;$

To estimate the area of the subset of $ADF$ that is covered using three consecutive elementary transformations, we shift the triangular pieces obtained each time so that the angle equal to the angle $ACB$ has a vertex at $D$, $A$ and $F$ respectively (after the second, third and fourth elementary transformation). The triangle $A'B'C'$ is shifted so that vertex $C'$ moves to $D$, $G$ moves to $F'$ and the angle $\angle FDF'$ is equal to the angle $\angle ACB$. We do the same with other triangles. See Figure \ref{parts} for the picture of the shifted pieces. This decreases the covered area, but does not change the angles between any of the lines.  

\begin{figure}[htbp]
  \begin{center}
\beginpgfgraphicnamed{triangles_pic13}
\begin{tikzpicture}[scale=0.6]
\filldraw[fill=black!25](22.2855,12.109) -- (13.8859,10.9033) -- (8.24393,16.0965) -- cycle;
\draw (8.28622,10.0995) -- (13.8859,10.9033)
(8.28622,10.0995) -- (8.24393,16.0965)
(8.28622,10.0995) -- (10.3145,14.1906)
(13.8859,10.9033) -- (9.42865,12.4039);
\fill (10.3145,14.1906) circle (1.5pt)
(8.24393,16.0965) circle (1.5pt)
(8.28622,10.0995) circle (1.5pt)
(9.42865,12.4039) circle (1.5pt)
(13.8859,10.9033) circle (1.5pt)
(22.2855,12.109) circle (1.5pt);
\draw[-stealth'] (13.8859,10.9033) ++(161.394:2) arc (161.394:188.169:2);
\draw[-stealth'] (8.28622,10.0995) ++(63.6293:1.5) arc (63.6293:90.404:1.5);
\draw[-stealth'] (8.24393,16.0965) ++(317.372:2) arc (317.372:344.147:2);
\draw (22.2855,12.109) node[right] {\small $F$}
(8.24393,16.0965) node[left] {\small $D$}
(8.28622,10.0995) node[left] {\small $A$}
(9.42865,12.4039) node[left] {\small $A'$}
(10.3145,14.1906) node[left] {\small $D'$}
(13.8859,10.9033) node[below] {\small $F'$};
    \end{tikzpicture}
\endpgfgraphicnamed

  \end{center}
  \caption{Covering using the angle $ACB$ (arrow)}
  \label{parts}
\end{figure}
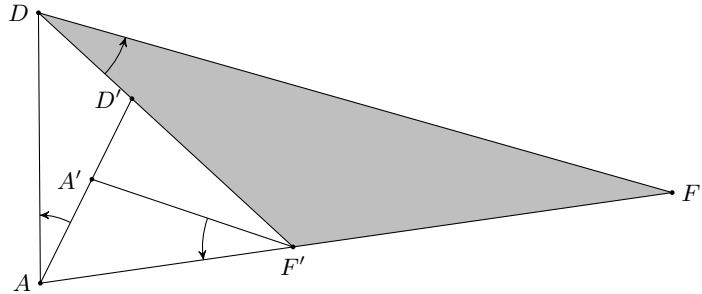

We have 
\begin{gather}
\angle A'F'D'+\angle AF'A'=\angle F'DF+\angle F'FD.
\end{gather}
This implies that the angle $\angle A'F'D'$ equals the angle $\angle AFD$. Similarly we show that other angles of the triangles $AFD$ and $A'F'D'$ are also equal. This means that three elementary transformations reduce the uncovered part to a triangle similar to $AFD$. The size of this triangle is not bigger than the size of $A'D'F'$. Therefore the area of the uncovered part is shrinking at a constant rate. Due to the difference in the area of $T'_\varepsilon$ and $T$ the procedure must end as soon as the area of the uncovered part is smaller than $\varepsilon$.

\subsubsection*{Step 4: Sequence of polarizations.}
$\;$

The presence of intersections in the second and all consecutive elementary transformations stops us from using the sequence of elementary transformations as a sequence of polarizations. 
To fix this problem we consider the same sequence of transformations, but in the reverse order. Each elementary transformation was applied to a triangle that is a part of the reflection of the triangle $ABC$. Therefore we can treat those elementary transformations as a transformations on $ABC$.   

As we remarked in \textit{Step 1} such transformations are equivalent to polarizations. We use the reflection lines used by the elementary transformations to produce polarizations.
Figures \ref{figl2r} and \ref{figl2r2} show the domain after the first two polarizations or the last two elementary transformations (outlined sets). The construction from \textit{Step 2} ensures that the triangular pieces introduced at every step have no intersection with the reflection lines. We also avoid self-intersecting since the triangular pieces are introduced in the reverse order. This means that we have a sequence of valid polarizations.

The last polarization (the first elementary transformation) fits the whole domain inside $T_\varepsilon'=DBE$ (see Figure \ref{figl2r3}). If the sequence is longer than on the example, we proceed in the same manner, building the spiral-like structure starting from the smallest inner triangle.

This proves that for arbitrary $\varepsilon>0$ there exists a finite sequence of 
polarizations, which transforms $T$ into a subset of $T_\varepsilon'$.

This sequence of polarizations can be constructed even for domains consisting of triangle $ABC$ and any set contained in the half-plane with boundary $AB$ and disjoint with $ABC$.

\begin{figure}[htbp]
  \begin{center}
    \subfloat[The first polarization (the third transformation)]{
\beginpgfgraphicnamed{triangles_pic14}
\begin{tikzpicture}[scale=0.26]
\clip (239.07,144.14) rectangle (256.65,165.24);
\fill[black!30] (206.645,317.241) -- (247.094,146.072) -- (295.547,208.722);
\fill[black!40] (89.829,221.742) -- (249.54,148.072) -- (178.508,113.04);
\draw (206.645,317.241) -- (246.926,146.783) -- (249.54,148.072) -- (248.877,148.378) -- (295.547,208.722)
(248.086,147.354) -- (368.55,0);
\fill (247.094,146.072) circle (3pt);
\draw (247.094,146.072) node[below] {\small $C$};
    \end{tikzpicture}
\endpgfgraphicnamed

  \label{figl2r}
}
\hspace{0.2cm}
\subfloat[The second polarization (the second transformation)]{
\beginpgfgraphicnamed{triangles_pic15}
\begin{tikzpicture}[scale=0.26]
\clip (239.07,144.14) rectangle (262.98,165.24);
\fill[black!20] (369.81,38.5379) -- (240.348,157.594) -- (318.697,169.18);
\fill[black!30] (206.645,317.241) -- (246.926,146.783) -- (249.54,148.072) -- (248.877,148.378) -- (295.547,208.722);
\draw (314.978,168.63) -- (244.235,158.169) -- (246.926,146.783) -- (249.54,148.072) -- (248.877,148.378) -- (249.498,149.18) -- (369.81,38.5379)
(245.488,152.867) -- (0,9.14765);
\fill (240.348,157.594) circle (3pt);
\draw (240.348,157.594) node [left] {\small $C$};
\end{tikzpicture}
\endpgfgraphicnamed

  \label{figl2r2}
}
\end{center}
  \caption{Reversed steps}
\end{figure}
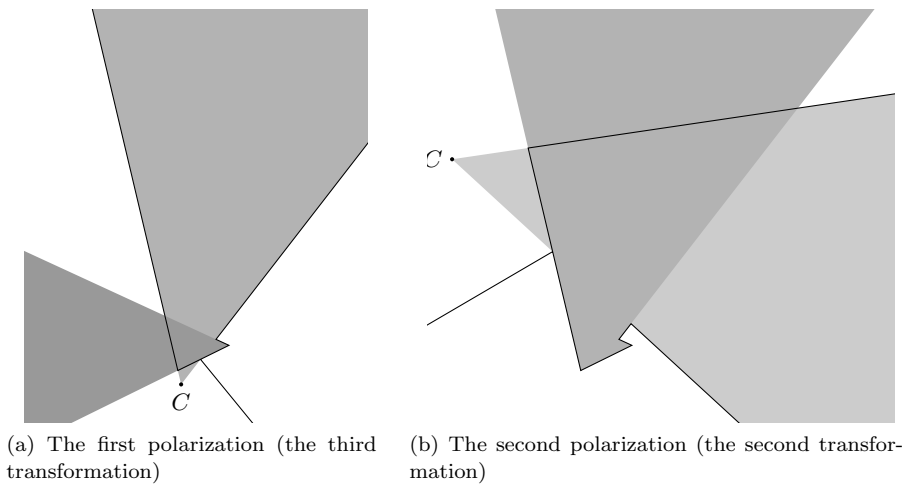

\begin{figure}[htbp]
  \begin{center}
\beginpgfgraphicnamed{triangles_pic16}
\begin{tikzpicture}[scale=0.12]
\clip (239.07,133.60) rectangle (323.45,186.33);
\fill[gray!30](318.697,169.18) -- (244.235,158.169) -- (246.926,146.783) -- (249.54,148.072) -- (248.877,148.378) -- (249.498,149.18) -- (369.81,38.5379);
\fill[gray!15](279.188,10.2571) -- (317.958,181.814) -- (246.926,146.782);
\draw (279.188,10.2571) -- (314.978,168.63) -- (244.235,158.169) -- (246.926,146.783) -- (249.54,148.072) -- (248.877,148.378) -- (249.498,149.18) -- (250.297,148.445) -- (246.926,146.782) -- (279.188,10.2571)
(281.054,163.614) -- (233.118,317.216);
\fill (244.235,158.169) circle (5pt)
(246.926,146.782) circle (5pt)
(279.188,10.2571) circle (5pt)
(314.978,168.63) circle (5pt)
(317.958,181.814) circle (5pt);
\draw (244.235,158.169) node[above left] {\small $D$}
(246.926,146.782) node[below left] {\small $A$}
(314.978,168.63) node[right] {\small $E$}
(317.958,181.814) node[right] {\small $C$};
      
    \end{tikzpicture}
\endpgfgraphicnamed

\end{center}
  \caption{The third polarization (the first transformation)}
  \label{figl2r3}
\end{figure}
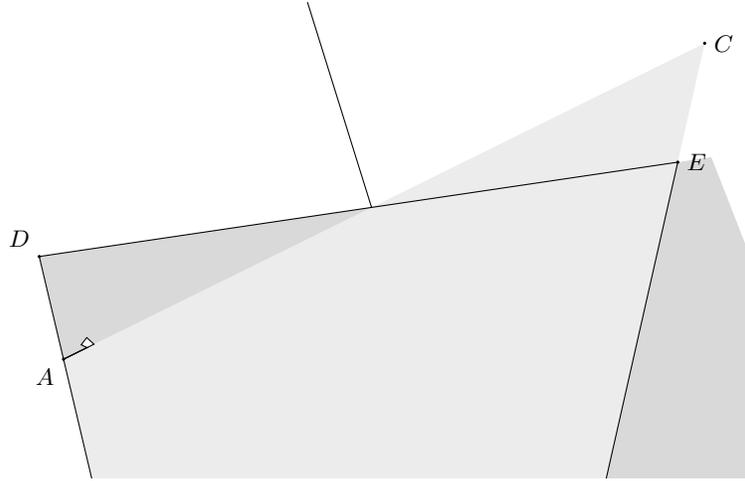

\end{proof}

\begin{remark}\label{anglesym}
  The same proof as for Lemma \ref{symtr} works for any domain contained in an infinite cone $ACB$ and containing $T$.
\end{remark}

\begin{lemma}\label{doublesym}
 Let $\Omega$ be a kite, that is a quadrilateral symmetric with respect to the longer diagonal and with perpendicular diagonals. Assume also that the length of the longer diagonal is smaller than the length of one of the sides. Given the fixed area and the smallest angle, the first eigenvalue is decreasing with the diameter. In particular, an isosceles triangle has a maximal first eigenvalue among kites with fixed area and the smallest angle, the kite consisting of two isosceles triangles has minimal first eigenvalue.
\end{lemma}
\begin{proof} 
  We can split an isosceles triangle $ABC$ into two right triangles and use Lemma \ref{symtr} and the remark above to symmetrize any one of those right triangles (for example $DBC$ on Figure \ref{figl2i}).
  Furthermore, we can perform all but the last polarization on both right triangles (Figure \ref{figl2i2} shows the domain at this stage of construction). Finally we can perform the last polarizations to get the quadrilateral with vertex $E$. The last transformations are indeed polarizations since the bisectors are cutting the line $BD$ below the point $B$ (certainly not between points $D$ and $E$).

  \begin{figure}[htb]
    \begin{center}
      \subfloat[one-sided symmetrization]{
\beginpgfgraphicnamed{triangles_pic17}
\begin{tikzpicture}[scale=0.6]
\clip (0,0) rectangle (10,10);
\filldraw[fill=black!30] (5.22013,0.96888) -- (8.97575,7.94063) -- (5.22013,8.66308) -- (5.22013,8.29295) -- (1.27472,8.29295);
\draw (7.14422,8.29295) -- (7.55805,12.6349)
(8.97575,7.94063) -- (5.22013,8.66308)
(5.22013,0) -- (5.22013,12.6349)
(1.27472,8.29295) -- (5.22013,0.96888)
(9.16554,8.29295) -- (1.27472,8.29295)
(5.22013,0.96888) -- (9.16554,8.29295);
\fill (5.22013,8.29295) circle (1.5pt)
(5.22013,0.96888) circle (1.5pt)
(9.16554,8.29295) circle (1.5pt)
(1.27472,8.29295) circle (1.5pt);
\draw (5.22013,0.96888) node[left] {\small $B$}
(5.22013,8.29295) node[above left] {\small $D$}
(1.27472,8.29295) node[left] {\small $A$}
(9.16554,8.29295) node[right] {\small $C$};
      \end{tikzpicture}
\endpgfgraphicnamed

    \label{figl2i}
}
\subfloat[two-sided symmetrization before the last step]{
\beginpgfgraphicnamed{triangles_pic18}
\begin{tikzpicture}[scale=0.7]
\clip (4,1) rectangle (10,9.3);
\filldraw[fill=black!30,draw=black!50] (0.96888,5.22013) -- (7.94063,8.97575) -- (7.928,9.04) -- (8.29295,9.1) -- (8.29295,1.34022) -- (7.928,1.3732) -- (7.94063,1.46451) -- cycle;
\draw (8.29295,7.14422) -- (12.6349,7.55805)
(8.29295,3.29604) -- (12.6349,2.88221)
(8.66308,5.22013) -- (7.94063,1.46451)
(7.94063,8.97575) -- (8.66308,5.22013);
\draw (4,5.22013) -- (12.6349,5.22013);
\fill (8.29295,9.16554) circle (1.5pt)
(8.29295,1.27472) circle (1.5pt)
(8.66308,5.22013) circle (1.5pt);
\draw (8.29295,1.27472) node [right] {\small $C$}
(8.66308,5.22013) node [below right] {\small $E$}
(8.29295,9.16554) node [right] {\small $A$};
\end{tikzpicture}
\endpgfgraphicnamed

    \label{figl2i2}
}
    \end{center}
    \caption{Isosceles triangle}
  \end{figure}
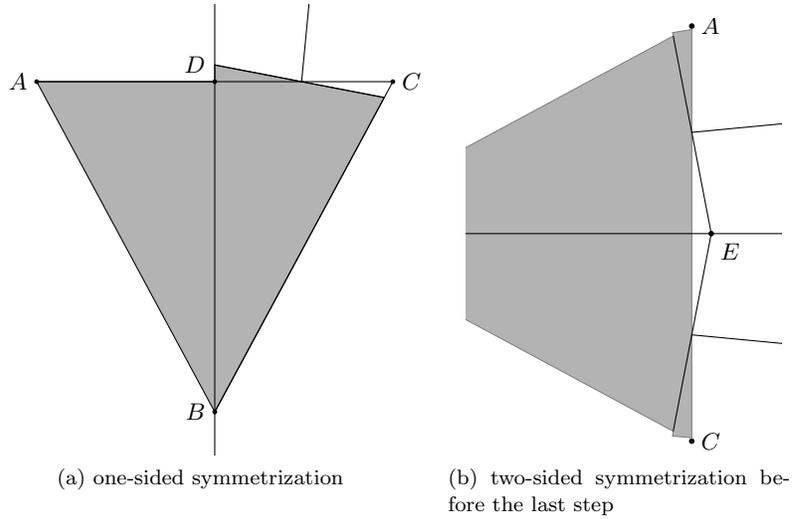

\end{proof}
The proof of the second inequality in Theorem \ref{lower2} requires a repeated use of the above results. 
As before we take a slightly larger sector 
$S(A+\varepsilon,\gamma)$ instead of $S(A,\gamma)$. We need to transform 
the isosceles triangle $I(A,\gamma)$  into a subset of $S(A+\varepsilon,\gamma)$. 
Using Lemma \ref{doublesym} we can symmetrize the triangle $ABC$ into a quadrilateral (see Figure \ref{figl2c}) with longer diagonal equal to the radius of the circular sector $S(A+\varepsilon,\gamma)$.

Now we divide both halves of the quadrilateral into $N$ (to be chosen later) parts by dividing the 
angle with vertex $B$ into equal parts (see Figure \ref{figl2d}).

\begin{figure}[htb]
  \centering
  \subfloat[quadrilateral]{
\beginpgfgraphicnamed{triangles_pic19}
\begin{tikzpicture}[scale=0.5]
\fill[black!30] (2.33976,4.24045) -- (9.31151,7.99606) -- (10.034,4.24045) -- (9.31151,0.48483);
\draw (9.66383,0.295037) -- (2.33976,4.24045)
(9.66383,8.18586) -- (9.66383,0.295037)
(2.33976,4.24045) -- (9.66383,8.18586)
(1,4.24045) -- (11,4.24045);
\draw (10.034,4.24045) arc (0:30:{10.034-2.33976})
(10.034,4.24045) arc (0:-30:{10.034-2.33976});
\fill (9.31151,0.48483) circle (1.5pt)
(10.034,4.24045) circle (1.5pt)
(9.31151,7.99606) circle (1.5pt)
(2.33976,4.24045) circle (1.5pt)
(9.66383,8.18586) circle (1.5pt)
(9.66383,0.295037) circle (1.5pt)
(9.66383,4.24045) circle (1.5pt);
\draw (2.33976,4.24045) node [below left] {\small $B$}
(9.66383,8.18586) node [right] {\small $A$}
(9.66383,0.295037) node [right] {\small $C$}
(9.66383,4.24045) node [below left] {\small $D$};
  \end{tikzpicture}
\endpgfgraphicnamed

  \label{figl2c}
  }
  \subfloat[subdivisions]{
\beginpgfgraphicnamed{triangles_pic20}
\begin{tikzpicture}[scale=0.5]
\clip (1,0.4) rectangle (11,8);
\fill[black!30] (1.94538,4.20101) -- (8.91713,7.95662) -- (9.63958,4.20101) -- (8.91713,0.445392);
\draw (1.94538,4.20101) -- (13.8691,5.68148)
(1,4.20101) -- (13.8691,4.20101)
(1.94538,4.20101) -- (12.6045,8.3425)
(1.94538,4.20101) -- (13.8691,7.20831);
\draw  (9.63958,4.20101) arc (0:30:{9.63958-1.94538})
(9.63958,4.20101) arc (0:-30:{9.63958-1.94538});
\fill (8.91713,0.445392) circle (1.5pt)
(9.40595,6.08265) circle (1.5pt)
(9.11725,6.98758) circle (1.5pt)
(1.94538,4.20101) circle (1.5pt)
(9.63958,4.20101) circle (1.5pt)
(8.91713,7.95662) circle (1.5pt)
(9.58095,5.14906) circle (1.5pt);
  \end{tikzpicture}
\endpgfgraphicnamed

  \label{figl2d}
  }
  \caption{Symmetrized isosceles triangle}
\end{figure}
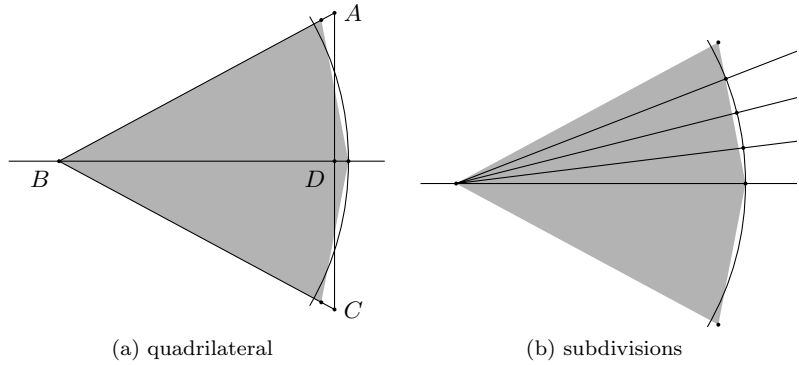

We work on each half separately. Consider a triangle formed by all but the inner most part. 
We use Remark \ref{anglesym} to symmetrize this triangle into a ``more isosceles'' triangle with one vertex on an intersection of a circular part and the innermost line. This makes the part outside of the circular sector smaller. Now we repeat this using decreased number of parts. The picture after two steps is shown on Figure \ref{figl2e}. By choosing a large enough $N$ and performing all steps we can fit the whole half of the quadrilateral inside of the circular sector. 

\begin{figure}[htb]
  \begin{center}
\beginpgfgraphicnamed{triangles_pic21}
\begin{tikzpicture}[scale=2]
\clip (-4.5,0) rectangle (0.5,2.2);
\fill[black!30] (-0.791005,-5.97013) -- (-4.44873,0.866002) -- (-2.66213,1.49309) -- (-2.64447,1.42267) -- (-1.73367,1.66611) -- (-1.71974,1.55324) -- (-0.791005,1.72408) -- (2.95696,1.03466);
\draw (-0.791005,0) -- (-0.791005,2.61468)
(-0.791005,-5.97013) -- (-1.85077,2.61468)
(-0.791005,-5.97013) -- (-2.94333,2.61468)
(3.13707,1.37128) -- (-0.791005,-5.97013)
(-1.73367,1.66611) -- (-1.71974,1.55324)
(-0.791005,-5.97013) -- (-4.10567,2.61468);
\draw  (-0.791005,1.72408) arc (90:120:{1.72408+5.97013})
(-0.791005,1.72408) arc (90:60:{1.72408+5.97013});
\fill (2.95696,1.03466) circle (0.5pt)
(-2.64447,1.42267) circle (0.5pt)
(-2.66213,1.49309) circle (0.5pt)
(-1.71974,1.55324) circle (0.5pt) 
(-3.5624,1.20763) circle (0.5pt)
(-4.44873,0.866002) circle (0.5pt)
(-0.791005,1.72408) circle (0.5pt)
(-1.73367,1.66611) circle (0.5pt);
    \end{tikzpicture}
\endpgfgraphicnamed

\end{center}
  \caption{Intermediate stage of construction}
  \label{figl2e}
\end{figure}
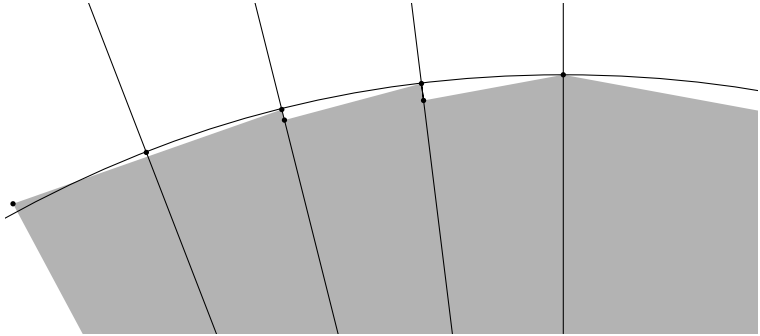

This ends the proof of the second inequality in the Theorem \ref{lower2}. The last thing to prove is the monotonicity property for the isosceles triangles.
Here we need to apply the continuous symmetrization to obtain the result. 
Suppose that we start with the isosceles triangle $I(A,\gamma)$ with the vertex angle smaller than
$\pi/3$. We want to show that if we increase the angle to $\pi/3\geq\gamma'>\gamma$ while the area
remains fixed, then the first eigenvalue decreases. First, we apply the continuous Steiner symmetrization with
respect to the line perpendicular to one of the arms. The length of the base increases, and its maximal length is obtained
when we reach the full Steiner symmetrization. Compare with Figure \ref{sym} where the shortest side of $\Omega^\alpha$ increases with $\alpha$.

If this maximum is smaller than the base of $I(A,\gamma')$, we use the Steiner symmetrization to get
an isosceles triangle with the angle $\gamma''$ between $\gamma$ and $\gamma'$ and the base equal to this maximum. We can repeat the above procedure on the triangle $I(A,\gamma'')$.

On the other hand, if the maximum is bigger than the base of $I(A,\gamma')$, then we can stop the continuous Steiner symmetrization at the time when the length of the
enlarged base of $I(A,\gamma)$ is equal to the base of $I(A,\gamma')$ (on Figure \ref{sym} we choose $\alpha$ such that the shortest side is equal to the base of $I(A,\gamma')$. 
Now, the Steiner symmetrization with respect to the base gives an isosceles triangle with the base equal to the base of $I(A,\gamma')$. This implies that it must be equal to $I(A,\gamma')$.

Suppose that the vertex angle $\gamma>\pi/3$. The same procedure as above shortens the base in this case. The same argument applies, but with maximum replaced with minimum. 
This ends the proof of the Theorem \ref{lower2}.

\section{The proofs of the upper bounds}\label{sup}
In this section we prove Theorem \ref{upper}. The approach we take is based on variational bounds
for the second eigenvalue. This approach is a modification of the similar method
used in \cite{S}. We start with the minimax formula for the second eigenvalue.

\begin{gather}\label{minimax}
  \lambda_2|_D=\inf_
  {f_1,f_2}
  \sup_{\alpha\in\Real}
  \frac{\int_D |\nabla (f_1+\alpha f_2)|^2}{\int_D |f_1+\alpha f_2|^2},
\end{gather}
where $f_1,f_2\in H^1_0(D)$ are linearly independent.
This formula is a special case of the general minimax formula for an arbitrary eigenvalue (see e.g. \cite{Da}).
As in \cite{S}, we will use known eigenfunctions for equilateral or right triangles
to obtain test functions for arbitrary triangles.

Consider the  equilateral triangle $T_e$ with vertices $(0,0)$, $(1,0)$ and $(1/2,\sqrt3/2)$.
The complete set of eigenfunctions is well known. For the exact formulas for these
eigenfunctions we refer the reader to \cites{P,M}. In particular, \cite{M} gives the simple formulas that
we will use in this paper. Let 
\begin{gather}
  z=\frac\pi3(2x-1),\label{zz}\\
  t=\pi\left(1-\frac2{\sqrt3} y\right).\label{tt}
\end{gather}
The first eigenfunction is given by
\begin{gather}
  \varphi_1(x,y)=(\cos(3z)-\cos(t))\sin(t).
\end{gather}

The second eigenvalue has multiplicity two. We will follow 
the notation from \cite{M} to name these eigenfunctions. All eigenfunctions can be divided into two kinds: symmetric and antisymmetric.
The first will be denoted by $S$ followed by two numbers identifying the eigenfunction. The second by $A$ also followed by two numbers. For the details about this notation we
refer the reader to \cite{M}. The two eigenfunctions belonging to the second eigenvalue are
\begin{gather}
  \varphi_{S21}(x)=\cos(4z)\sin(2t)+\cos(5z)\sin(t)-\cos(z)\sin(3t),\\
  \varphi_{A21}(x)=\sin(4z)\sin(2t)+\sin(5z)\sin(t)-\sin(z)\sin(3t).
\end{gather}

Let $T$ be an arbitrary triangle. We can assume that one side of this triangle
is equal to the segment from $(0,0)$ to $(1,0)$, and that the last vertex $(u,v)$ is in the upper halfspace. 
Then, there exists a unique linear transformation $L$ from $T$ onto $T_e$.
As in \cite{S}, we will compose $L$ with the eigenfunctions of $T_e$ to obtain suitable test functions for $T$. 

Using formula (\ref{minimax}) we obtain an upper bound
\begin{gather}\label{genbound}
  \lambda_2|_T\leq \sup_{\alpha\in\Real} 
  \frac{\int_T |\nabla (g_\alpha\circ L)|^2}{\int_T |g_\alpha\circ L|^2},
\end{gather}
where $g_\alpha$ is a linear combination of two known eigenfunctions.

As the first two test functions we can take
\begin{gather}\label{test1}
  g^1_\alpha(x,y)=\varphi_{S21}(x,y)+\alpha\varphi_{S11}(x,y),\\
  g^2_\alpha(x,y)=\varphi_{A21}(x,y)+\alpha\varphi_{S11}(x,y).\label{test2}
\end{gather}

If the triangle $T$ is almost equilateral, we can expect that its eigenvalues
and eigenfunctions are similar to the eigenvalues and eigenfunctions of $T_e$.
Also, the linear transformation $L$ should not perturb the bound (\ref{genbound}) 
significantly. Therefore, we can expect that for almost equilateral triangles
this should be a good upper bound for $\lambda_2|_T$.

Notice that the linear combinations in (\ref{test1}) and (\ref{test2}) consist of two orthogonal 
functions. Hence the second norm of this combination is just the sum of the 
second norms. Therefore 

\begin{gather}
  \sup_{\alpha\in\Real} 
  \frac{\int_T |\nabla (g_\alpha\circ L)|^2}{\int_T |g_\alpha\circ L|^2}
  =
  \frac{a\alpha^2+b\alpha+c}{e\alpha^2+f},
\end{gather}
where $a$,$c$,$e$,$f$ are strictly positive.
This rational function has a limit $a/e$, as $\alpha\to \pm\infty$. By taking its derivative we can also find the critical points
\begin{gather}
  \alpha_1=\frac1{bd}\left(af-ce-\frac12\sqrt{4b^2ef+(2af-2ce)^2}\right),\\
  \alpha_2=\frac1{bd}\left(af-ce+\frac12\sqrt{4b^2ef+(2af-2ce)^2}\right).
\end{gather}
If we evaluate the function at those points and simplify, we get
\begin{gather}
  C_1=\frac1{2ef}\left(af+ce-\frac12\sqrt{4b^2ef+(2af-2ce)^2}\right),\\
  C_2=\frac1{2ef}\left(af+ce+\frac12\sqrt{4b^2ef+(2af-2ce)^2}\right).
\end{gather}
The expression under the root is clearly nonnegative. It is zero
if $b=0$ and $af=ce$. In such case $C_1=C_2=a/e$, hence the maximum
is $a/e$. If the expression under the root is positive, then we have two
distinct critical points, and $C_2>C_1$. This means that this rational function
has an absolute maximum $C_2$ and absolute minimum $C_1$.

This leads to a new bound for the eigenvalue
\begin{gather}\label{upsimple}
  \lambda_2|_T\leq C_2=\frac1{2ef}\left(af+ce+\frac12\sqrt{4b^2ef+(2af-2ce)^2}\right).
\end{gather}

To finish the proof of Theorem \ref{upper} we need to use one of the lower
bounds for the first eigenvalue. Due to the simple form of the bound, we will use
Freitas's result (\ref{freitas}). To prove the first bound in Theorem \ref{upper} we
need to show that
\begin{gather}\label{bbound1}
  ((\ref{upsimple})-(\ref{freitas}))R^2\leq\frac{16\pi^2}{27},
\end{gather}
where we have used the reference number of the equation for the value given by it.
The number on the right side of the inequality is the exact value obtained for the
equilateral triangle.

The second bound will be proved if we can show that
\begin{gather}\label{bbound2}
  \frac{(\ref{upsimple})}{(\ref{tA})}\leq \frac73.
\end{gather}
Notice that this time we use P\'olya's isoperimetric bound (\ref{tA}).

Let us begin with the proof of (\ref{bbound1}). We first use $g^1_\alpha$ as
a test function. The expressions (\ref{upsimple}) and (\ref{freitas}) can be written in terms of
the vertices of the triangle $T$. But we assumed that the vertices are
$(0,0)$, $(1,0)$ and $(u,v)$. We denote the lengths of the sides of the triangle 
by $1$, $M=\sqrt{u^2+v^2}$, $N=\sqrt{(1-u)^2+v^2}$ and we can assume that $N\geq M\geq 1$. Then 
the bound (\ref{bbound1}) (in terms of the lengths of the sides) is equivalent to
\begin{gather}
0\geq -1612800N^2(1+M+N)^2\pi^2+27\Bigl\{-413343N^2V\nonumber\\
+11200(9(M^2-1)^2+2(M^2+1)N^2+20N^4)\pi^2\label{long}\\
+N^2\sqrt{655128046899V^2-74071065600VW
\pi^2+8028160000W^2\pi^4}\Bigr\},\nonumber
\end{gather}
where $V=M^2+N^2-2$ and $W=M^2+N^2+1$.

The expression is quite complicated since the integrals of the function $g^1_\alpha$ are very
cumbersome to calculate. This task could be accomplished by hand since this function is
just a sum of the products of the trigonometric functions. However, symbolic calculations in
Mathematica are used to obtain this expression in a short time. The script performing all the
calculations from this section is included in the last section.

One can check numerically where this inequality is true. 
If we put $U=N-M$ then
the set of all possible triangles can be characterized by $U\in[0,1)$ and 
$M\geq1$. The dark part of Figure \ref{fignum} corresponds to those triangles for which
the inequality is valid.

\begin{figure}[htb]
  \centering
\beginpgfgraphicnamed{triangles_pic22}
\begin{tikzpicture}[scale=7]
\fill[black!50,smooth]  plot coordinates {(.94512,.01923)(.91814,.09053)(.88924,.1599)(.85841,.22734)(.82661,.29382)(.79289,.35837)(.75725,.421)(.71871,.48073)(.67728,.53758)(.63296,.59153)(.58382,.64067)(.52794,.68306)(.46146,.71485)(.37571,.72352)(.30442,.69655)(.25046,.65223)(.20614,.59827)(.16953,.53661)(.13774,.47013)(.11076,.39884)(.08764,.32369)(.06644,.24661)(.04717,.1676)(.03176,.08475)(.01986,.01923)};
\fill[black!20] (.85233,.01923) -- (.85233,.21066) -- (.08408,.21066) -- (.08408,.01923);
  \draw[stealth'-stealth'] (0.01923,0.87) -- (0.01923,0.01923) -- (1.05,0.01923);
  \foreach \x in {1,1.1,1.2,1.3,1.4}{
    \draw ({(\x-1)*2.15+0.01923},0.03) -- ({(\x-1)*2.15+0.01923},0.01) node[below] {\tiny \x};}
  \foreach \x in {0,0.02,0.04,0.06,0.08,0.1,0.12}{
    \draw (0.03,\x*6.3+0.01923) -- (0.01,\x*6.3+0.01923) node[left] {\tiny \x};}
  \draw (1,0.02) node[below] {\tiny $M$};
  \draw (0.02,0.85) node[left] {\tiny $U$};
  \end{tikzpicture}
\endpgfgraphicnamed

  \caption{Bound 1/Case 1: The numerical solution}
  \label{fignum}
\end{figure}
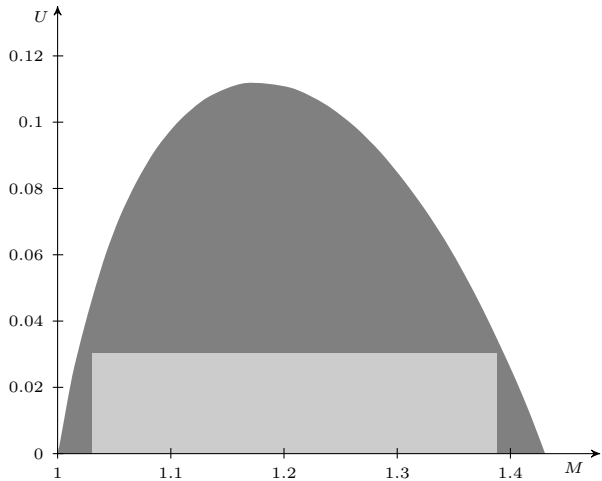

  It is clear that not all triangles can be handled this way. Hence we have to 
use more then one test function, and divide all triangles into subregions of
$U\in[0,1)$ and $M\geq1$. We will define other test functions at the end of the section.

We now prove inequality (\ref{long}) on the gray rectangle shown in Figure \ref{fignum}. 
More precisely we take $U\in[0,0.03]$ and $M\in[1.03,1.39]$. 

The inequality (\ref{long}) can be written as
\begin{gather}
  P(N,M)+Q(N,M)\sqrt{R(N,M)}\leq 0,
\end{gather}
where $P$, $Q$ and $R$ are polynomials in $N$ and $M$.
It will be proved, if we can show that
\begin{gather}
P(N,M)\leq0,\\
Q^2(N,M)R(N,M)-P^2(M,N)\leq0.
\end{gather}
This is a system of polynomial inequalities. Unfortunately the degrees of those 
polynomials are 4 and 8 in each variable. Therefore there is almost no hope to 
solve this system using any conventional method. 

Instead, we developed 
an algorithm for solving such polynomial inequalities on rectangles. 
The next section contains a detailed description of this algorithm and the proof 
of its correctness. It turned out that in our case this method gives the proof
of the inequality for any test function we tried.

To finish the proof of the first part of Theorem \ref{upper} we need to define the other test functions
and rectangles for each corresponding inequality.

We already have two test functions (\ref{test1}) and (\ref{test2}) that came from the equilateral triangle. 
We can also use the first two eigenfunctions of the half of the equilateral triangle. That is, of
the right triangle with angle $\pi/6$. Its eigenfunctions are 
certain antisymmetric eigenfunctions of the equilateral triangle. Hence, we get the third test function
\begin{gather}
g^3_\alpha(x,y)=\varphi_{A31}(x,y)+\alpha\varphi_{A21}(x,y),
\end{gather}
where (using notation (\ref{zz}) and (\ref{tt}))
\begin{gather}
\varphi_{A31}(x,y)=\sin{5z}\sin{3t}-\sin{2z}\sin{4t}-\sin{7z}\sin{t}.
\end{gather}
In this case we also need to get a new linear transformation $L'$ that transforms 
the given triangle into the right triangle.

The last case when all eigenfunctions are known is the right isosceles triangle.
Since this triangle is ``half'' of a square, its eigenfunctions are equal
to eigenfunctions of the square with diagonal nodal lines. We get
\begin{gather}
g^4_\alpha(x,y)=\phi_2(x,y)+\alpha\phi_1(x,y),
\end{gather}
where
\begin{gather}
  \phi_1(x,y)=\sin{2\pi x}\sin{\pi y}+\sin{\pi x}\sin{2\pi y},\\
\phi_2(x,y)=\sin{3\pi x}\sin{\pi y}-\sin{\pi x}\sin{3\pi y}.
\end{gather}

We can also mix the eigenfunctions from the different triangles provided we
use appropriate linear transformation for each of them. In this manner we obtain
the last (fifth) case needed to prove the theorem by taking $g^1_\alpha(x,y)$ and $\frac12\phi_2(x,y)$.

Each of the five cases requires a rectangle on which we can prove the bound.
We take
\begin{enumerate}
\item $U\in[0,0.03]$, $M\in[1.03,1.39]$,
\item $U\in[0,0.2]$, $M\in[1,1.03]$,
\item $U\in[0,1)$, $M\in[1.39,\infty)$,
\item $U\in[0.2,1)$, $M\in[1,1.39]$,
\item $U\in[0.03,0.2]$, $M\in[1.03,1.39]$.
\end{enumerate}
Notice that these rectangles exactly cover the infinite strip $[0,1)\times[1,\infty)$. Since the sides of the triangle are
$N\geq M\geq 1$ and $U=N-M$, the strip includes all possible combinations of lengths of the sides.

The proof in \textit{cases (3)-(5)} is exactly the same as in \textit{case (1)}. Additional step has to
be performed in \textit{case (2)}. Consider a triangle $T'$ similar to $T$ with sides $1$,
$N'=M/N$, $M'=1/N$. We have $1\geq N'\geq M'$. Consider inequality (\ref{bbound1}) 
for $T'$. Note that the diameter of $T'$ used in (\ref{freitas}) is now $1$
($N$ in other cases). Just like before we get an expression similar to (\ref{long}) but
involving $N'$ and $M'$. After a change of variable $M'\to N'-M'+1$ and 
$N'\to 1-U'$ we can apply our algorithm with a rectangle given in \textit{case (2)}.
This proves the bound (\ref{bbound1}) for $T'$, and hence for $T$ since the bound
is invariant under scaling.  

We need to check that the transformation
\begin{gather}
\frac{M}{M+U}=1-U',\\
\frac1{M+U}=2-U'-M',
\end{gather}
changes the rectangle $U\in[0,0.2)$, $M\in[1,1.03)$ into a subset of the same
rectangle in primed variables. From the first equation we get
\begin{gather}
0\leq U'=1-\frac{M}{M+U}\leq 1-\frac1{1.23}<0.19.
\end{gather}
From the second equation
\begin{gather}
1\leq M'=\frac{M-1}{M+U}+1\leq \frac{M-1}M+1=2-\frac1M\leq 2-\frac1{1.03}<1.0292.
\end{gather}

This ends the proof of the first part of Theorem \ref{upper}. The second part requires an additional argument. We still need to use test functions $g^1_\alpha(x,y)$ through $g^5_\alpha(x,y)$ on the following rectangles 
\begin{enumerate}
  \item $U\in[0,0.09]$, $M\in[1,1.37]$;
  \item not needed;
  \item $U\in[0,0.42]$, $M\in[1.37,2.05]$;
  \item $U\in[0.09,0.2]$, $M\in[1,1.37]$;
  \item $U\in[0.2,0.42]$, $M\in[1,1.37]$.
\end{enumerate}
Notice that in each case $U\leq 0.42$. Since the triangles are acute we have $N^2\leq M^2+1$, hence $U\leq \sqrt2-1$. Also, all these rectangles cover only the cases with $M\leq 2.05$.

Therefore we need a circular sector type--bound similar to the one in \cite{S}. We use the sector bound from Theorem \ref{lower2} for the first eigenvalue and the upper bound for the second eigenvalue based on the biggest circular sector contained in the given triangle. Suppose that we have an acute triangle with area $A$ and smallest angle $\gamma$. Let $N\geq M\geq2.05$ be the two longest sides. Since the triangle is acute, the largest sector contained in the triangle has the radius equal to the longest altitude $H$. We get
\begin{gather}
  \left. \frac{\lambda_2}{\lambda_1}\right|_{Triangle}\leq
  \frac{\lambda_2|_{S(\gamma H^2/2,\gamma)}}{\lambda_1|_{S(A,\gamma)}}.
\end{gather}

This bound is certainly good enough, but it is hard to deal with due to the complicated formula for the
radius of the circular sector $S(A,\gamma)$. Therefore we have to rely on the weaker bound
\begin{gather}\label{sectorbound}
  \left.\frac{\lambda_2}{\lambda_1}\right|_{Triangle}\leq
  \frac{\lambda_2|_{S(\gamma H^2/2,\gamma)}}{\lambda_1|_{S(\gamma MN/2,\gamma)}}.
\end{gather}
This bound also follows from Theorem \ref{lower2} by taking the smallest sector containing the isosceles triangle $I(A,\gamma)$ (first inequality in Theorem \ref{lower2}) with the arms of the length $\sqrt{MN}$. 

The eigenvalues of the circular sectors are given in terms of the zeros $j_{v,n}$ of the Bessel function of index $v$. Here $n$ indicates $n$-th smallest zero. We have
\begin{gather}
  \lambda_1|_{S(\gamma R^2/2,\gamma)}=R^{-2} j_{\pi/\gamma,1}^2,\\
  \lambda_2|_{S(\gamma H^2/2,\gamma)}\leq\lambda_{12}(Sector)=R^{-2} j_{\pi/\gamma,2}^2.
\end{gather}
Notice that we only have the inequality for the second eigenvalue due to the presence of another eigenfunction $\lambda_{21}(Sector)$. This eigenvalue may be smaller than $\lambda_{12}(Sector)$.

We can use the bounds for the zeros of the Bessel functions proved in \cite{QW}. That is,
\begin{gather}
  v-\frac{a_k}{\sqrt[3]{2}}\sqrt[3]{v}<j_{v,k}<v-\frac{a_k}{\sqrt[3]{2}}\sqrt[3]{v}+\frac3{20}a_k^2\frac{\sqrt[3]{2}}{\sqrt[3]{v}},
\end{gather}
where $a_k$ are zeros of the Airy function with $a_1\approx-2.3381$ and $a_2\approx-4.0879$.

Using the last two facts we obtain
\begin{gather}
  \begin{split}  
    \left. \frac{\lambda_2}{\lambda_1}\right|_{Triangle}&\leq
    \frac{NM}{H^2}\left(\frac{x^3-\frac{a_2}{\sqrt[3]{2}}x+\frac3{20}a_2^2\frac{\sqrt[3]{2}}{x}}{x^3-\frac{a_1}{\sqrt[3]{2}}{x}}\right)^2
    \\&=
    \frac{NM}{H^2}\left(\frac{1-\frac{a_2}{\sqrt[3]{2}}x^{-2}+\frac3{20}a_2^2\frac{\sqrt[3]{2}}{x^4}}{1-\frac{a_1}{\sqrt[3]{2}}{x^{-2}}}\right)^2,
  \end{split} 
\end{gather}
where $x=(\pi/\gamma)^{1/3}$. Note that for acute triangles we have $N^2\leq M^2+1$. We need to consider two cases based on the length of $H$. First assume that
$H^2\geq M^2-1/16$. This is equivalent to the condition that the altitude $H$ divides the shortest side (with the length $1$) into two parts with one of them not longer than $1/4$. Using the inequality $N^2\leq M^2+1$, we obtain 
\begin{gather}
  \frac{NM}{H^2}\leq  \frac{M\sqrt{M^2+1}}{M^2-\frac1{16}}=\sqrt{1+M^{-2}}\left(1-\frac1{16M^2}\right)^{-1}=:z_1(M).
\end{gather}

If $M^2-1/4\leq H^2\leq M^2-1/16$ then the altitude $H$ divides the shortest side into two parts, one of length $\delta$ satisfying $1/4<\delta<1/2$ and the other of length $1-\delta$. Then
\begin{gather}
  N^2=H^2+(1-\delta)^2=M^2-\delta^2+(1-\delta)^2=M^2+1-2\delta\leq M^2+\frac12.
\end{gather}
This gives
\begin{gather}
  \frac{NM}{H^2}\leq  \frac{M\sqrt{M^2+\frac12}}{M^2-\frac14}=\sqrt{1+\frac1{2M^2}}\left(1-\frac1{4M^2}\right)^{-1}=:z_2(M).
\end{gather}
Let $z(M)=\max\left\{ z_1(M),z_2(M) \right\}$. Note that $z(M)$ is decreasing since both $z_1(M)$ and $z_2(M)$ are decreasing. Making the substitution $y=x^{-2}$ we get
\begin{gather}\label{My}
  \begin{split}  
    \left.\frac{\lambda_2}{\lambda_1}\right|_{Triangle}&\leq
    z(M)\left(\frac{1-\frac{a_2}{\sqrt[3]{2}}y+\frac{3\sqrt[3]{2}}{20}a_2^2y^2}{1-\frac{a_1}{\sqrt[3]{2}}y}\right)^2
    \\&=
    z(M)\left(c_1+c_2y+\frac{c_3}{1+c_4 y}\right)^2,
  \end{split}
\end{gather}
where the constants $c_i$ satisfy
\begin{gather}
  c_1=\frac{(10a_1-3a_2)a_2}{10a_1^2}\approx0.83133>0,\\
  c_2=-\frac{3a_2^2}{10\sqrt[3]{2}a_1}\approx1.70183>0,\\
  c_3=\frac{10a_1^2-10a_1a_2+3a_2^2}{10a_1^2}\approx0.16867>0,\\
  c_4=-\frac{a_1}{\sqrt[3]{2}}\approx1.85575>0.
\end{gather}

We want to show that the right hand side of (\ref{My}) is decreasing with $M$ (note that $y$ also depends on $M$ since it depends on $\gamma$). First we can show that this expression is increasing with $y$ for a fixed $M$. Indeed, the derivative with respect to $y$ is $c_2-\frac{c_3}{(1+c_4 y)^2}$ and it is positive for $y>0$ due to the condition $c_2>c_3$. This means that the expression is also increasing in $\gamma$ since $y=(\pi/\gamma)^{-2/3}$. If we fix $M$, then from all triangles with sides $1$, $M\geq1$, $N\geq M$, the isosceles triangle ($M=N$) has the biggest angle $\gamma$. In this case $\cos \gamma=1-\frac1{2M^2}$. As a result we get an upper bound for the ratio of the first two eigenvalues in terms of $M$. That is, we have
\begin{gather}
  \left.\frac{\lambda_2}{\lambda_1}\right|_{Triangle}\leq
  z(M)\left(c_1+c_2y_M+\frac{c_3}{1+c_4y_M}\right)^2,
\end{gather}
with $y_M=(\arccos(1-(2M^2)^{-1})/\pi)^{2/3}$. But $y_M$ and $z(M)$ are decreasing with $M$, hence the right hand side is decreasing with $M$. To finish the proof we just need to check that this is smaller than $7/3$, as required in (\ref{bbound2}), for $M=2.05$ (we get $\approx2.3285<7/3$). This ends the proof of Theorem \ref{upper}.

\section{An algorithm for polynomial inequalities}\label{salg}
This section gives the algorithm for proving polynomial inequalities in two
variables with arbitrary degrees. The domain we deal with is a rectangle. 
The proof of the correctness is also given.

Suppose that we have an inequality
\begin{gather}\label{poly}
P(x,y)=\sum_{i=0}^n\sum_{j=0}^m c_{i,j}x^i y^j\leq0,
\end{gather}
for $x\in(0,a)$ and $y\in(0,b)$. Any other rectangle can be shifted to the origin,
hence reducing the problem to this case.

The idea behind the algorithm is very simple. For any monomial
\begin{gather}
  c_{i,j}x^iy^j\leq c_{i,j}\min\{a x^{i-1} y^j, b x^i y^{i-1}\},\mbox{ if }c_{i,j}>0,\label{posup}\\
  c_{i,j}x^iy^j\leq c_{i,j}\max\{a^{-1} x^{i+1} y^j, b^{-1} x^i y^{i+1}\},\mbox{ if }c_{i,j}<0\label{negup}.
\end{gather}
We can use this simple observation to reduce the number of positive coefficients in 
$P(x,y)$. Clearly, if we apply any of the above inequalities finite number
of times on any of the monomials in P(x,y), we obtain an upper bound for
$P(x,y)$. If we can reduce the whole polynomial to 0, we proved
inequality (\ref{poly}). 

We need to describe a sequence of reductions, that leads to a constant or to a 
polynomial with only positive coefficients. If we get at least one strictly positive coefficient we cannot say that
inequality is false, since we are using an upper bound for $P(x,y)$.
Thus, the algorithm does not always work although it is possible to generalize it.
In case of the false answer we can divide the rectangle into four identical 
sub--rectangles, and rerun the algorithm on each of them. As long as the inequality
$P(x,y)\leq0$ is strict, this recursive procedure should give the proof. 

It is worth noting that for 8 out of the 9 polynomials we have in the previous section this
method works on the whole rectangle (without sub--dividing). In the case of the third test function
in the second part of Theorem \ref{upper} we need to split the given rectangle into halves.
The gray rectangle on Figure \ref{fignum} in the previous section is almost as big as possible given that the inequality 
is true only for the dark points. Hence the method works very well in this case.
By running the script from the last section one can see that in all other cases
the method is also very efficient.

Now we will define the optimal sequence of reductions. The only way to reduce 
a positive coefficient is by lowering the power of one of the variables (using (\ref{posup})). Similarly,
to reduce a negative coefficient we have to increase one of the powers (using (\ref{negup})). Each time,
two of the coefficients combine giving a new, possibly negative coefficient. Write
\begin{gather}
P(x,y)=\sum_{i=0}^n x^i Q_i(y),
\end{gather}
where 
\begin{gather}
Q_i(y)=\sum_{j=0}^m c_{i,j}y^j. 
\end{gather}

To avoid ambiguity, we start with $Q_n$. Any negative coefficient in $Q_n$ 
can be used only to combine with some positive coefficient with higher power of $y$.
It is impossible to use them to interact with $Q_i$ for $i<n$.
Therefore we inductively (starting from $j=0$) check if $c_{n,j}<0$ and in case this is true
we get an upper bound 
\begin{gather}
c_{n,j}x^n y^j+c_{n,j+1}x^n y^{j+1}\leq (c_{n,j}b^{-1}+c_{n,j+1})x^ny^{j+1},
\end{gather}
and we redefine $c_{n,j+1}=c_{n,j}b^{-1}+c_{n,j+1}$. If the last coefficient turns
out to be negative, we can just change it to $0$.

As a result we changed $Q_n(y)$ into a polynomial $Q_n'(y)$ with nonnegative coefficients
(possibly all equal to $0$). A positive coefficient can only be altered by lowering 
one of the powers. We could lower the power of $y$ but, ultimately, if we want to 
obtain a constant as a final bound we have to also lower the power of $x$ in all
coefficients of $x^n Q_n'(y)$. Hence we get an upper bound
\begin{gather}
x^nQ_n'(y)+x^{n-1}Q_{n-1}(y)\leq (a Q_n'(y)+Q_{n-1}(y))x^{n-1}, 
\end{gather}
and we redefine $Q_{n-1}(y)=aQ_n'(y)+Q_{n-1}(y)$. This approach guarantees that
any negative coefficient of $Q_{n-1}(y)$ can be used to reduce as many positive
coefficients from $Q_n(y)$ as possible.

Now we repeat the whole procedure for $Q_{n-1}(y)$ and for all others by induction.
At the end we get $Q_0'(y)$ which has only nonnegative coefficients. If any of those
coefficients is strictly positive, the algorithm failed.
But $Q_0(y)=0$ means that $P(x,y)\leq0$ on the rectangle $(0,a)\times(0,b)$.

The implementation of this algorithm is a part of the script in the following section. The following simple example
shows how the algorithm works. Let
\begin{gather}
  P(x,y)=x^2y^2-x^2y+2xy^2+x^2+xy+y^2-3x-2y.
\end{gather}
We want to show that $P(x,y)\leq0$ on $(0,1)\times(0,1)$. We have 
\begin{gather*}
  Q_2(y)=1-y+y^2,\\
  Q_1(y)=-3+y+2y^2,\\ 
  Q_0(y)=-2y+y^2.
\end{gather*}

First, we have $Q_2(y)\leq 1+0+(y^2-y^2)$ (by rising the power of $-y$). Now, we redefine $Q_1(y)=(-3+1)+y+2y^2$ and we get the bound
$Q_1(y)\leq 0+(y-2y)+2y^2\leq0+0+(2y^2-y^2)$. Hence, $Q_0(y)=0-2y+(y^2+y^2)$ and as the last step $Q_0(y)\leq 0+0+(2y^2-2y^2)=0$. 
This shows that the inequality is true.

\section{Mathematica package for variational bounds}

The package TrigInt.m contains functions helpful in finding variational upper bounds. All functions have short explanations available (use \lstinline{?Fun} in Mathematica). This package requires Mathematica 6.0. Although all functions are written for triangular domains, one could triangulate any polygon and still use the package.

The function \lstinline{TrigInt} is equivalent to Integrate, but it is much faster for linear combinations of trigonometric functions. It is necessary due to a very slow integration of some trigonometric functions in Mathematica 5.1 and above. 

\lstloadlanguages{Mathematica}
\lstset{
	language=Mathematica,
	showstringspaces=false,
	breaklines,
	basicstyle={\scriptsize},
	numbers=left,
	numberstyle={\tiny},
	morekeywords={Cross,TrigReduce,AffineTransform,Del,Total,FullSimplify,RegionPlot,NotebookDirectory,Alpha},
	deletekeywords={N},
	stringstyle={\textit},
	escapeinside={*@}{@*}
}
\begin{lstlisting}
(* ::Package:: *)

BeginPackage["TrigInt`"]

TrigInt::usage =
        "TrigInt is faster than Integrate for large trigonometric functions.";

Equilateral::usage=
		"Eigenfunctions of Neumann and Dirichlet Laplacian on the equilateral triangle with vertices (0,0), (1,0), (1/2,Sqrt[3]/2):
		 Equilateral[Dirichlet,Symmetric] [m,n]  -  1<=m<=n
		 Equilateral[Dirichlet,Antisymmetric] [m,n]  -  1<=m<n
		 Equilateral[Neumann,Symmetric] [m,n]  -  0<=m<=n
		 Equilateral[Neumann,Antisymmetric] [m,n]  -  0<=m<n
		 Equilateral[Eigenvalue] [m,n]
Eigenfunctions of the right triangle with vertices (0,0), (1,0), (0,Sqrt[3]):
		 Equilateral[Dirichlet,Half] [m,n]  -  1<=m<n
		 Equilateral[Neumann,Half] [m,n]  -  0<=m<=n
		 Equilateral[Eigenvalue,Half] [m,n]
Vertices:
		 Equilateral[]
		 Equilateral[Half]";

Square::usage=
		"Eigenfunctions of Neumann and Dirichlet Laplacian on the square with vertices (0,0), (1,0), (1,1), (0,1):
		 Square[Dirichlet] [m,n]  -  m>=1,n>=1
		 Square[Neumann] [m,n]  -  m>=0,n>=0
		 Square[Eigenvalue] [m,n]
Eigenfunctions of the right isosceles triangle with vertices (0,0), (1,0), (0,1):
		 Square[Dirichlet,Half] [m,n]  -  1<=m<n
		 Square[Neumann,Half] [m,n]  -  0<=m<=n
Vertices:
		 Square[Half]";

Rayleigh::usage="Rayleigh quotient for functions on triangular domains. 
Many functions can be specified, each having more than one domain, but all domains should be disjoint.

Rayleigh[f1,T11,T12,...,T1n,f2,T21,T22,...,T2m,...]";

Transplant::usage="Linear change of coordinates.

Transplant[Function,TargetTriangle,InitialTriangle]";

Limits::usage="Generates limits for integration over a triangle with one side contained in the x-axis.";

Del::usage="Gradient with respect to x and y.";
Grad::usage="Gradient with respect to x and y.";

Area::usage="Area of a triangle with given vertices.";
Perimeter::usage="Perimeter of a polygon with given vertices.";
T::usage="T[a,b] - triangle with vertices (0,0), (1,0) and (a,b)";

Begin["`Private`"]

(* TrigInt *)
TrigInt[f_,x__]:=Expand[TInt[TrigReduce[f],x]];
TInt[f_,x__,y_]:=TInt[TInt[f,y],x];
TrigInt::nomatch="No match for `1`. Integrate used.";
(* integrals *)
psin[n_]:=psin[n]=Evaluate[Integrate[y^n Sin[#1+#2 y],{y,#3,#4}]]&;
pcos[n_]:=pcos[n]=Evaluate[Integrate[y^n Cos[#1+#2 y],{y,#3,#4}]]&;
pp[n_]:=pp[n]=Evaluate[Integrate[y^n,{y,#1,#2}]]&;
(* substitutions *)
sub[x_,a_,b_]:={
X_. x^n_. Sin[A_.+B_. x]/;FreeQ[{X,A,B,n},x]:>(X psin[n][A,B,a,b]),
X_. x^n_. Cos[A_.+B_. x]/;FreeQ[{X,A,B,n},x]:>(X pcos[n][A,B,a,b]),
X_. Sin[A_.+B_. x]/;FreeQ[{X,A,B},x]->(X psin[0][A,B,a,b]),
X_. Cos[A_.+B_. x]/;FreeQ[{X,A,B},x]->(X pcos[0][A,B,a,b]),
X_. x^n_./;FreeQ[{X,n},x]:>(X pp[n][a,b]),
X_/;FreeQ[{X},x]->X (b-a),
X_:>(Message[TrigInt::nomatch,X];Integrate[X,{x,a,b}])
};
(* single integral *)
TInt[f_,{x_,a_,b_}]:=(
ff=f/.Sin[X_]:>Sin[Collect[X,{x},Simplify]]/.Cos[X_]:>Cos[Collect[X,{x},Simplify]];
ff=Expand[ff];
ff=ff+f0+f1;
ff=Replace[ff,sub[x,a,b],1];
ff/.f0->0/.f1->0
);

(* eigenfunctions *)
h=1;
r=h/(2Sqrt[3]);
u=r-Global`y;
v=Sqrt[3]/2(Global`x-h/2)+(Global`y-r)/2;
w=Sqrt[3]/2(h/2-Global`x)+(Global`y-r)/2;
EqFun[f_,g_]:=f[Pi (-#1-#2)(u+2r)/(3r)]g[Pi (#1-#2)(v-w)/(9r)]+
	f[Pi #1 (u+2r)/(3r)]g[Pi (2#2+#1)(v-w)/(9r)]+
	f[Pi #2 (u+2r)/(3r)]g[Pi (-2#1-#2)(v-w)/(9r)];
Equilateral[Global`Neumann,Global`Symmetric]=Evaluate[Simplify[EqFun[Cos,Cos]]]&;
Equilateral[Global`Neumann,Global`Antisymmetric]=Evaluate[Simplify[EqFun[Cos,Sin]]]&;
Equilateral[Global`Dirichlet,Global`Symmetric]=Evaluate[Simplify[EqFun[Sin,Cos]]]&;
Equilateral[Global`Dirichlet,Global`Antisymmetric]=Evaluate[Simplify[EqFun[Sin,Sin]]]&;
Equilateral[Global`Eigenvalue]=Evaluate[4/27(Pi/r)^2(#1^2+#1 #2+#2^2)]&;
Equilateral[Global`Neumann,Global`Half]=Equilateral[Global`Neumann,Global`Symmetric]/.Global`x->(1-Global`x)/2/.Global`y->Global`y/2//Simplify;
Equilateral[Global`Dirichlet,Global`Half]=Equilateral[Global`Dirichlet,Global`Antisymmetric]/.Global`x->(1-Global`x)/2/.Global`y->Global`y/2//Simplify;
Equilateral[Global`Eigenvalue,Global`Half]=Evaluate[1/27(Pi/r)^2(#1^2+#1 #2+#2^2)]&;
Equilateral[]=T[1/2,Sqrt[3]/2];
Equilateral[Global`Half]=T[0,Sqrt[3]];

Square[Global`Dirichlet]=Evaluate[Sin[#1 Pi Global`x]Sin[#2 Pi Global`y]]&;
Square[Global`Neumann]=Evaluate[Cos[#1 Pi Global`x]Cos[#2 Pi Global`y]]&;
Square[Global`Eigenvalue]=Evaluate[Pi^2(#1^2+#2^2)]&;
Square[Global`Dirichlet,Global`Half]=Evaluate[Sin[#1 Pi Global`x]Sin[#2 Pi Global`y]-(-1)^(#1+#2)Sin[#2 Pi Global`x]Sin[#1 Pi Global`y]]&;
Square[Global`Neumann,Global`Half]=Evaluate[Cos[#1 Pi Global`x]Cos[#2 Pi Global`y]+(-1)^(#1+#2)Cos[#2 Pi Global`x]Cos[#1 Pi Global`y]]&;
Square[Global`Half]=T[0,1];

(* Transplantation *)
LT[{p1_,p2_,p3_},{q1_,q2_,q3_}]:=Module[{ff},
ff[x_]:={x.{aa,bb}+cc ,x.{dd,ee}+ff};
AffineTransform[{{{aa,bb},{dd,ee}},{cc,ff}}]/.Solve[{ff[p1]==q1,ff[p2]==q2,ff[p3]==q3},{aa,bb,cc,dd,ee,ff}][[1]]
];
ST[p_,q_]:=Thread[{Global`x,Global`y}->LT[p,q][{Global`x,Global`y}]];
Transplant[f_,T1_,T2_]:=f/.ST[T1,T2];

(* Rayleigh quotient *)
Rayleigh[p__]:=Num[p]/Denom[p];
Num[f_,T__List]:=Total[GInt[f,#]&/@{T}];
Num[f_,Longest[T__List],p__]:=Num[f,T]+Num[p];
Denom[f_,T__List]:=Total[NInt[f,#]&/@{T}];
Denom[f_,Longest[T__List],p__]:=Denom[f,T]+Denom[p];
GInt[f_,T_]:=TrigInt[Del[f].Del[f],Limits[T]];
NInt[f_,T_]:=TrigInt[f^2,Limits[T]];

Limits[{{c_,0},{d_,0},{a_,b_}}]:=Sequence[{Global`y,Min[0,b],Max[0,b]},{Global`x,Min[c,d]+(a-Min[c,d])Global`y/b,Max[c,d]+(a-Max[c,d])Global`y/b}];
Limits[{{c_,0},{d_,0},{a_,b_},cond_}]:=Sequence@@Refine[{{Global`y,Min[0,b],Max[0,b]},{Global`x,Min[c,d]+(a-Min[c,d])Global`y/b,Max[c,d]+(a-Max[c,d])Global`y/b}},cond];
(*
Limits[{{c_,0},{d_,0},{a_,b_}}]:=Sequence[{Global`y,0,b},{Global`x,c+(a-c)Global`y/b,d+(a-d)Global`y/b}];
*)
Del[f_]:={D[f,Global`x],D[f,Global`y]};
Grad[f_]:=Del[f];
T[a_,b_]:={{0,0},{1,0},{a,b}};
T[a_,b_,c_]:={{0,0},{1,0},{a,b},c};
Area[{p1_,p2_,p3_}]:=Abs[Cross[Append[p2-p1,0],Append[p3-p1,0]]][[3]]/2//Simplify;
Perimeter[l_]:=Total[Sqrt[#.#]&/@(l-RotateLeft[l])]//Simplify;

End [ ]

EndPackage [ ]
















\end{lstlisting}

\section{Script for triangles}
In this section we give a script that performs all the calculations from Sections \ref{sup} and \ref{salg}. It is important to note that all operations are done symbolically, thus
there are no numerical errors in any calculation. This allows us to use the script as a part of our proofs.

The script has many comments to help the reader follow the code. The output of the script also contains many values representing different stages of calculations. The meaning of each value is explained in the comments.

The script handles both bounds from Theorem \ref{upper}. It can be executed either inside Mathematica GUI or using a command line 
\begin{lstlisting}[numbers=none]
MathKernel -run "bound=_;<<script.m"
\end{lstlisting}
where the value of the variable ``bound'' should be either 1 or 2. The package TrigInt should be in the same folder as the script. The script generates all cases considered in the proof. Each case corresponds
to one of the test functions. For the second bound, the second test function is not needed, hence
we used it to split the rectangle for the third test function. This is the only case when the algorithm fails to
give a proof for a whole rectangle. 

Note that long lines are split into multiple lines and continuation is indented.

\begin{lstlisting}
(* bound: 1 or 2 *)
If[! ValueQ[bound], bound = 1];

AppendTo[$Path,NotebookDirectory[]];
<< TrigInt`
(** test functions **)
(* the equilateral triangle *)
f[1] = Equilateral[Dirichlet, Symmetric][1, 1];
f[S21] = Equilateral[Dirichlet, Symmetric][1, 2];
f[A21] = Equilateral[Dirichlet, Antisymmetric][1, 2];
(* the right isosceles triangle *)
g[1] = Square[Dirichlet, Half][2, 1];
g[2] = Square[Dirichlet, Half][3, 1];
(* the half of the equilateral triangle *)
h[1] = Equilateral[Dirichlet, Half][1, 2];
h[2] = Equilateral[Dirichlet, Half][1, 3];
(** find maximum of a rational function **)
ff = (c[1] + c[2] Alpha + c[3] Alpha^2)/(d[1] + d[3] Alpha^2);
ff = ff /. Solve[D[ff, Alpha] == 0, Alpha] // FullSimplify;
ff = ff[[2]];
(** algorithm for solving inequalities **)
(** implementation of Section *@\ref{salg}@* **)
CumFun[f_, l_] := Rest[FoldList[f, 0, l]];
PolyNeg[P_, {x_, y_}, {dx_,dy_}] := ((Fold[CumFun[Min[#1, 0]/dy + #2 &, Map[Max[#1, 0] &, #1] dx + #2] &, 0, Reverse[CoefficientList[P, {x, y}]]] // Max) <= 0);
 
Do[
  Print["case: ", case];
  (**case 1**)
  fun[1] = Transplant[Alpha f[1] + f[S21], T[a, b], Equilateral[]];
  (*rectangles for each bound*)
  r[1][1] = {{103/100, 36/100, 1, 1.45}, {0, 3/100, 0, 0.15}};
  r[2][1] = {{100/100, 37/100, 1, 1.5}, {0, 9/100, 0, 0.5}};
  (**case 2**)
  fun[2] = Transplant[Alpha f[1] + f[A21], T[a, b], Equilateral[]];
  r[1][2] = {{1, 3/100, 1, 1.5}, {0, 2/10, 0, 0.3}};
  (**case 3**)
  fun[3] = Transplant[Alpha h[1] + h[2], T[a, b], Equilateral[Half]];
  r[1][3] = {{139/100, Infinity, 1, 2}, {0, 1, 0, 1}};
  (*two different rectangles for the second bound*)
  r[2][3] = {{137/100, 63/100, 1, 2.5}, {0/100, 22/100, 0, 0.5}};
  If[case == 2 && bound == 2, case = 3; 
   r[2][3] = {{137/100, 63/100, 1, 2.5}, {22/100, 20/100, 0, 0.5}}];
  (**case 4**)
  fun[4] = Transplant[Alpha g[1] + g[2], T[a, b], Square[Half]];
  r[1][4] = {{1, 39/100, 1, 2}, {1/5, 4/5, 0, 1}};
  r[2][4] = {{1, 37/100, 1, 1.5}, {20/100, 22/100, 0, 0.5}};
  (**case 5**)
  fun[5] = fun[1] + 1/2 Transplant[g[2], T[a, b], Square[Half]];
  r[1][5] = {{103/100, 36/100, 1, 1.5}, {3/100, 17/100, 0, 0.25}};
  r[2][5] = {{100/100, 37/100, 1, 1.5}, {9/100, 11/100, 0, 0.5}};
  (**test if 0 on the sides of the triangle**)
  (*output 1-3:three zeros if boundary conditions are met*)
  fun[case] /. y -> 0 // FullSimplify // Print;
  fun[case] /. y -> b x/a // FullSimplify // Print;
  fun[case] /. y -> b (1 - t) /. x -> a - (a - 1) t // FullSimplify // Print;
  (**bounds for the eigenvalues**)
  frac = Simplify[b^2 Rayleigh[fun[case], T[a, b,b>0]], {b > 0, a <= 1/2}] // Together;
  (**test if integral has a correct form, rational function**)
  (*output 4:number of the coefficients of the numerator*)
  (cN = CoefficientList[Numerator[frac], Alpha]) // Dimensions // Print;
  (*output 5:coefficients of the denominator*)
  (cD = CoefficientList[Denominator[frac], Alpha]) // Print;
  frac = ff /. c[i_] :> cN[[i]] /. d[i_] :> cD[[i]];
  lambda2 = frac/b^2;(*upper bound*)
  diam = If[case != 2, N, 1];(*diameter*)
  R = b/(1 + M + N);(*inradius*)
  lambda1[1] = Pi^2 (4/(diam^2) + (diam^2)/(M^2 - a^2)); (*lower bound: Freitas*)
  lambda1[2] = 8 Pi^2/Sqrt[3]/b; (*lower bound: Polya*)
  rM = r[bound][case][[1]];
  rU = r[bound][case][[2]];
  bd[1] := FullSimplify[27/16/Pi^2 (lambda2 - lambda1[1]) R^2 /. b -> Sqrt[M^2 - a^2] /. a -> (M^2 - N^2 + 1)/2];
  bd[2] := FullSimplify[3 lambda2/7/lambda1[2] /. b -> Sqrt[M^2 - a^2] /. a -> (M^2 - N^2 + 1)/2];
  bd2 = Numerator[bd[bound]]^bound - Denominator[bd[bound]]^bound;
  bd3 = If[case != 2, bd2, bd2 /. {M -> 1/N, N -> M/N}] /. N -> U + M;
  RegionPlot[{bd3 < 0 /. N -> U + M, rM[[1]] <= M <= rM[[1]] + rM[[2]] && rU[[1]] <= U <= rU[[1]] + rU[[2]]}, {M, rM[[3]], rM[[4]]}, {U, rU[[3]], rU[[4]]}] // Print;
  (**change inequality into polynomials**)
  bd2 = If[case != 2, bd2 /. N -> U + M, bd2 /. M -> N - M + 1 /. N -> 1 - U];
  bd2 = bd2 /. M -> M + rM[[1]] /. U -> U + rU[[1]] // ExpandAll;
  p = Position[bd2, Sqrt[_?(! NumberQ[#] &)]];(*root expression*)
  P1 = ReplacePart[bd2, 0, p] // ExpandAll;(*polynomial part*)
  P2 = (bd2 - P1)^2 // ExpandAll;(*expression under the root*)
  (*output 7-8:check if we have a polynomial*)
  PolynomialQ[P1, {M, U}] // Print;
  PolynomialQ[P2, {M, U}] // Print;
  (**proofs**)
  (*output 9-10:symbolic solution*)
  PolyNeg[P1, {U, M}, {rU[[2]], rM[[2]]}] // Print;
  PolyNeg[P2 - P1^2, {U, M}, {rU[[2]], rM[[2]]}] // Print;
,{case, 5}];

Exit[]
\end{lstlisting}

\end{document}